\documentclass[twoside]{amsart}
\usepackage[cp1251]{inputenc}
\usepackage[T2A]{fontenc}
\usepackage{array}
\usepackage{amssymb}
\usepackage{amsmath}
\usepackage{amsthm}
\usepackage{latexsym}
\usepackage{bm}
\usepackage{enumerate}
\usepackage[dvips]{graphicx}
\usepackage{epsf}
\usepackage{wrapfig}
\usepackage{euscript}
\usepackage[english,russian]{babel}
\usepackage{textcomp}
\usepackage{setspace}
\newtheorem{theorem}{Theorem}
\newtheorem{lemma}{Lemma}
\newtheorem{definition}{Definition}
\newtheorem{proposition}{Proposition}

\newtheorem{corollary}{Corollary}
\begin{document}
{\selectlanguage{english}
\binoppenalty = 10000 %
\relpenalty   = 10000 %

\pagestyle{headings} \makeatletter

\renewcommand{\@oddhead}{\raisebox{0pt}[\headheight][0pt]{\vbox{\hbox to\textwidth{{Lindstr\"{o}m theorem for bi-intuitionistic propositional logic}\hfill \strut\thepage}\hrule}}}
\makeatother

\title{Maximality of bi-intuitionistic propositional logic}

\author{Grigory Olkhovikov}

\address{ Department of Philosophy I\\
Ruhr University Bochum\\
Bochum, Germany }

\address{ Department of Philosophy \\
Ural Federal University\\
Ekaterinburg, Russia }

\email{grigory.olkhovikov@rub.de, grigory.olkhovikov@gmail.com}

\author{Guillermo Badia}

\address{ Institute of Philosophy and Scientific Method \\
	Johannes Kepler University\\
	Linz, Austria
}

\address{
	School of Historical and Philosophical Inquiry\\
	University of Queensland\\
	Brisbane, Australia
}
\email{guillebadia89@gmail.com}
\date{}
\maketitle

\begin{abstract}

In the style of Lindstr\"om's theorem for classical first-order logic, this article characterizes propositional bi-intuitionistic  logic as the
maximal (with respect to expressive power) abstract logic
satisfying a certain form of compactness, the Tarski union
property and preservation under bi-asimulations. 
Since  bi-intuitionistic logic introduces new complexities in the intuitionistic setting by adding the analogue of a backwards looking modality, the present paper constitutes a non-trivial modification of previous work  done by the authors for intuitionistic logic \cite{baok}. 
\bigskip

\emph{Keywords:} Lindstr\"{o}m theorem, bi-intuitionistic logic, abstract model theory, bi-asimulations.

\bigskip

\emph{Math subject classification:}     03C95,          03B55.
\end{abstract}
\section{Introduction}

In a series of papers from the 1970s, Cecylia Rauszer \cite{rau1, rau2, rau3} studied an extension of intuitionistic logic obtained by adding the algebraic dual of the intuitionistic implication $\rightarrow$ to the language. This co-implication connective, which we denote in this paper by ``$\ll $'',  is sometimes also called \emph{subtraction} \cite{re}. In the Kripke semantics for bi-intuitionistic logic, it behaves similarly to a backwards looking diamond modality $\Diamond^{-1}$. The resulting logic is known as \emph{Heyting-Brouwer} or \emph{Bi-intuitionistic logic}. In recent years, the study of this very natural logic has received some degree of attention from different scholars \cite{ba1, ba2,  go,  gr, ko, o3,  pi, sk, tr}. Our own work in this field has focused on the semantic study of bi-intuitionistic logic. The present contribution continues to develop this line of research, in particular, by studying a Lindstr\"om theorem for this logic.

Per Lindstr\"om famously provided a battery of results \cite{lin, lin2, lin3, lin4} which characterized classical first-order logic as the maximal ``abstract logic" (or ``model-theoretic language" \cite{fe}) satisfying certain combinations of properties from the following list: Compactness, the L\"owenheim-Skolem property, the Robinson consistency property, the Karp property, the Tarski Union property, the abstract Completeness property and the $\lambda$-omitting types property. This opened up the possibility of studying logics (involving cardinality quantifiers, infinitary disjunctions and conjunctions, second order quantification, etc.) in accord with what their model theory could offer and situating them in a much broader abstract model-theoretic framework. The textbook references for this subject  (called ``abstract model theory") are \cite{bar1, barfer}.  

Lindstr\"om theorems have also been studied in the context of non-classical logics \cite{mm}, with modal logic providing perhaps the most fruitful case \cite{otto, vanB, rijke, enqvist}. In the context of propositional intuitionistic logic, the authors of the present paper obtained such a theorem in \cite{baok} using the Tarski union property, a form of compactness and preservation under asimulations (inspired by work done for modal logic in \cite{enqvist}). Since bi-intuitionistic logic is such a natural companion to intuitionistic logic, the question immediately arose whether a similar theorem would hold for bi-intuitionistic logic replacing asimulations (introduced in \cite{o}) by an appropriate bi-intuitionistic analogue (such as that in \cite{baok}). 

In this article we will  adapt the somewhat delicate arguments from \cite{baok} to a bi-intuitionistic context. One of the most substantial difficulties of this task is the backwards looking modality, making the unraveling construction, which plays a central role in our proof, more complex. Hence, to use it the way we need it requires new special care.\footnote{For instance, the theory encoding all the types of a given model, which is often a pivotal piece in the proofs of this kind, can no longer be given by a finite number of formula schemes; we had to use an inductive definition instead.} Finally, the paper is self-contained and it will not require any familiarity with \cite{baok} on behalf of the reader (which, of course, means that we will repeat even the parts from \cite{baok}  which are rather easily adaptable, and not just the non-trivial parts).

The article is arranged as follows: in \S \ref{S:Prel} we establish the main bits of our
notation and define some basic concepts. \S \ref{S:unravel}
is then devoted to proving some less immediate lemmas on
bi-intuitionistic unravellings and bi-intuitionistically saturated
models. In \S \ref{S:Abstract} we define the notion of an
abstract logic in the context of extensions of bi-intuitionistic logic and formulate our main result. In \S \ref{main} we provide the proof of said result.  We end the paper by summarizing our work and 
 suggesting some lines of further inquiry.

\section{Preliminaries}\label{S:Prel}

We start this section by fixing some notational conventions to be used in this paper. For any $n > 0$, we will denote by $\bar{o}_n$ the sequence
$(o_1,\ldots, o_n)$ of objects of any kind; moreover, somewhat
abusing the notation, we will sometimes denote $\{o_1,\ldots,
o_n\}$ by $\{\bar{o}_n\}$. We will denote by
$(\bar{o}_n)^{\frown}(\bar{r}_m)$ the
concatenation of $\bar{o}_n$ and $\bar{r}_m$. 
The ordered $1$-tuple will
be identified with its only member, the ordered $0$-tuple will be denoted by $\Lambda$, and the last element of a sequence $\alpha$ will be denoted by $end(\alpha)$ so that we will have, for example, that $end(\bar{o}_n)= o_n$.

If $f$ is any function, then we will denote by $dom(f)$ its domain
and by $rang(f)$ the image of $dom(f)$ under $f$; if $rang(f)
\subseteq M$, we will also write $f: dom(f) \to M$.
Finally, if $X, Y$ are sets,
then we will write $X \Subset Y$, if $X \subseteq Y$ and $X$ is
finite.

In this paper, we consider the language of bi-intuitionistic
propositional logic, which we identify with its set of formulas.
This language is generated from some set of propositional letters
by a finite number of applications of connectives from the set $\{
\bot, \wedge, \vee, \to, \ll \}$; in this set, $\ll$ stands for bi-intuitionistic co-implication and the connectives are assumed to
have their usual arities. The set of propositional letters can be
in general arbitrarily large, but we assume that it is disjoint
from the above set of connectives (and from the set of logical
symbols of every logic which we are going to consider below). Any
set with this property we will call \emph{signature}.
Bi-intuitionistic propositional formulas will be denoted with Greek
letters like $\varphi$, $\psi$ and $\theta$,\footnote{We will be
adjoining to them subscripts and superscripts when needed, and the
same is assumed for any other notations introduced in this paper.}
whereas the elements of signatures will be denoted by letters
$p$, $q$, and $r$. If $\Theta$ is a signature, then $BIL(\Theta)$
denotes the set of those bi-intuitionistic propositional formulas
which only contain propositional letters from $\Theta$.

For this language, we assume the standard Kripke semantics.
The typical notations for bi-intuitionistic Kripke models that we are
going to use below are as follows:
$$
\mathcal{M} = \langle W, \prec, V\rangle, \mathcal{M}' = \langle W',
\prec', V'\rangle, \mathcal{M}_n = \langle W_n, \prec_n, V_n\rangle,
$$
$$
\mathcal{N} = \langle U, \lhd, Y\rangle, \mathcal{N}' = \langle U',
\lhd', Y'\rangle,\mathcal{N}_n = \langle U_n, \lhd_n, Y_n\rangle,
$$
where $n \in \omega$. As we proceed, we will also define some
operations on bi-intuitionistic Kripke models which are going to
affect the notation.

If $\Theta$ is a signature and $\mathcal{M}$ is a bi-intuitionistic
Kripke $\Theta$-model,\footnote{Bi-intuitionistic Kripke models, as they are defined in this paper, are the same as Kripke models used for intuitionistic propositional logic. Therefore, one can also call them intuitionistic Kripke models. The reason for our preferred terminology is that in this paper we only consider these models in connection with bi-intuitionistic logic and its extensions.} then $W$ is a non-empty set of
\emph{worlds}, \emph{states}, or \emph{nodes}, $\prec$ is a partial order
on $W$ called $\mathcal{M}$'s \emph{accessibility relation}, and
$V$ is the evaluation function for $\Theta$ in $\mathcal{M}$, that
is to say, a function $V:\Theta \to 2^W$ such that for
every $p \in \Theta$ and arbitrary $s, t \in W$, it is true that:
$$
w\mathrel{\prec}v \Rightarrow (w \in V(p) \Rightarrow v \in V(p)).
$$
As  usual, we denote the reduct of a $\Theta$-model
$\mathcal{M}$ to a smaller signature $\Sigma \subseteq\Theta$ by
$\mathcal{M}\upharpoonright\Sigma$.
Next, we need a definition of isomorphism between models:
\begin{definition}\label{D:isomorphism}
	{\em Let $\mathcal{M}$, $\mathcal{N}$ be $\Theta$-models. A  function $g: W \to U$ is called an \emph{ isomorphism from $\mathcal{M}$ onto $\mathcal{N}$} (write $g:\mathcal{M} \cong \mathcal{N}$)
		iff $g$ is a bijection, and for all $v,u \in W$ and $p \in \Theta$ it is true that:
		\begin{align}
		&v\mathrel{\prec}u \Leftrightarrow g(v)\mathrel{\lhd}g(u)
		\label{E:ic1}\tag{\text{i-rel}}\\
		&v \in V(p) \Leftrightarrow g(v) \in Y(p)\tag{\text{i-atom}}
		\end{align}
	}
\end{definition}

If $\mathcal{M}$ and $\mathcal{N}$ are two bi-intuitionistic Kripke
$\Theta$-models then we say that $\mathcal{M}$ is a
\emph{submodel} of $\mathcal{N}$ and write $\mathcal{M} \subseteq
\mathcal{N}$ iff $W \subseteq U$, $R = S\upharpoonright(W\times
W)$ and, for every $p \in \Theta$, $V(p) = Y(p) \cap W$. In
general, for every $W \subseteq U$ there exists a corresponding
submodel $\mathcal{M}$ of $\mathcal{N}$ with $W$ as its universe. Moreover, such an  $\mathcal{M}$ is unique up to an isomorphism. In such cases we may also denote $\mathcal{M}$ by
$\mathcal{N}(W)$.

If $\mathcal{M}_1 \subseteq,\ldots, \subseteq \mathcal{M}_n
\subseteq,\ldots$ is a countable chain of bi-intuitionistic Kripke
models then the model:
$$
\bigcup_{n \in \omega}\mathcal{M}_n = (\bigcup_{n \in \omega}W_n,
\bigcup_{n \in \omega}\prec_n, \bigcup_{n \in \omega}V_n)
$$
is again a bi-intuitionistic Kripke model.

A \emph{pointed} bi-intuitionistic Kripke $\Theta$-model is a pair of
the form $(\mathcal{M}, w)$ such that $w \in W$. The class of all pointed bi-intuitionistic Kripke
models in all signatures (resp. in the signature $\Theta$) will be
denoted by $Pmod$ (resp. by $Pmod_\Theta$).

In this paper, we
are not going to consider any non-bi-intuitionistic Kripke models.
Therefore, we will omit the qualification `bi-intuitionistic Kripke'
in what follows and will simply speak about (pointed)
$\Theta$-models. We assume the standard satisfaction relation for
$BIL(\Theta)$:
\begin{align*}
&\mathcal{M}, w \models_{BIL} p \Leftrightarrow w \in V(p),
&&p \in \Theta;\\
&\mathcal{M}, w \models_{BIL} \varphi \wedge \psi \Leftrightarrow
\mathcal{M}, w \models_{BIL} \varphi\textup{ and } \mathcal{M}, w
\models_{BIL} \psi;\\
&\mathcal{M}, w \models_{BIL} \varphi \vee \psi \Leftrightarrow
\mathcal{M}, w \models_{BIL} \varphi\textup{ or } \mathcal{M}, w
\models_{BIL} \psi;\\
&\mathcal{M}, w \models_{BIL} \varphi \to \psi \Leftrightarrow
(\forall v\mathrel{\succ}w )(\mathcal{M}, v
\not\models_{BIL} \varphi\textup{ or } \mathcal{M}, v
\models_{BIL} \psi);\\
&\mathcal{M}, w \models_{BIL} \varphi \ll \psi \Leftrightarrow
(\exists v\mathrel{\prec}w)(\mathcal{M}, v\models_{BIL} \varphi\textup{ and }\mathcal{M}, v
\not\models_{BIL} \psi);\\
&\mathcal{M}, w \not\models_{BIL} \bot.
\end{align*}
Note that the formulas of $BIL(\Theta)$ get satisfied at pointed
$\Theta$-models rather than at models alone.

Since classical negation is not available in bi-intuitionistic logic,
it makes sense to define bi-intuitionistic theories, in the same way as intuitionistic theories are often defined, that is to say as pairs of sets of formulas including truth assumptions and falsehood
assumptions of a given theory (see e.g.
\cite[p. 110]{GabbayMaksimova} w.r.t. intuitionistic first-order
logic). Thus a $BIL(\Theta)$-theory becomes a pair $(\Gamma,
\Delta) \in 2^{BIL(\Theta)}\times 2^{BIL(\Theta)}$, where formulas
in $\Gamma$ are assumed to be true and formulas from $\Delta$ are
assumed to be false. If $(\mathcal{M}, w)$ is a pointed
$\Theta$-model, then we define $Th_{BIL}(\mathcal{M}, w)$, the
\emph{$BIL(\Theta)$-theory of} $(\mathcal{M}, w)$, as follows:
$$
Th_{BIL}(\mathcal{M}, w) := (\{ \varphi \in BIL(\Theta)\mid
\mathcal{M}, w \models_{BIL} \varphi \}, \{ \varphi \in
BIL(\Theta)\mid \mathcal{M}, w \not\models_{BIL} \varphi \}).
$$
We also introduce a special notation for the left and right
projection of $Th_{BIL}(\mathcal{M}, w)$, that is to say, for the
\emph{positive} and for the \emph{negative} part of this theory,
denoting them by $Th^+_{BIL}(\mathcal{M}, w)$ and
$Th^-_{BIL}(\mathcal{M}, w)$, respectively. Inclusion of
intuitionistic theories must then involve set-theoretic inclusion
of their respective projections, so that we define:
$$
(\Gamma, \Delta) \subseteq (\Gamma', \Delta') \Leftrightarrow
\Gamma \subseteq \Gamma'\textup{ and }\Delta \subseteq \Delta'.
$$
Also the set theoretic operations on the pairs of sets will be understood componentwise, e.g. we will assume that $(\Gamma, \Delta) \cap (\Gamma', \Delta') = (\Gamma\cap \Gamma', \Delta \cap \Delta')$ and similarly for $\cup$.

It is clear then that a $BIL(\Theta)$-theory $(\Gamma, \Delta)$ is
\emph{$BIL$-satisfiable} iff we have \[
(\Gamma, \Delta) \subseteq
Th_{BIL}(\mathcal{M}, w)
\]
 for some
$(\mathcal{M}, w)\in Pmod_\Theta$. In this case we will also write $\mathcal{M},
w \models_{BIL} (\Gamma, \Delta)$. If $\mathcal{M} \subseteq
\mathcal{N}$ and for every $w \in W$ it is true that
$Th_{BIL}(\mathcal{M}, w) = Th_{BIL}(\mathcal{N}, w)$, then we say
that $\mathcal{M}$ is an $BIL$\emph{-elementary submodel} of
$\mathcal{N}$ and write $\mathcal{M} \preccurlyeq_{BIL}
\mathcal{N}$. A standard generalization of the notion of a $BIL$-elementary submodel leads as to the concept of $BIL$-elementary embedding:
\begin{definition}\label{D:embedding}
	{\em Let $\mathcal{M}$, $\mathcal{N}$ be $\Theta$-models. An injective function $g: W \to U$ is called a \emph{ $BIL$-elementary embedding of
			$\mathcal{M}$ into $\mathcal{N}$} iff
		for all $v,u \in W$ it is true
		that
		\begin{align}
		&v\mathrel{\prec}u \Leftrightarrow g(v)\mathrel{\lhd}g(u)
		\label{E:c1}\tag{\text{rel}}\\
		&Th_{BIL}(\mathcal{M}, v) = Th_{BIL}(\mathcal{N}, g(v))\label{E:c2}\tag{\text{theories}}
		\end{align}
	}
\end{definition}
We collect some of the easy consequences of our series of definitions in the following lemma:
\begin{lemma}\label{L:embedding}
	Let $\mathcal{M}$, $\mathcal{N}$ be $\Theta$-models and let $\Sigma \subseteq \Theta$. Then the following statements hold:
	\begin{enumerate}
		\item $Th_{BIL}(\mathcal{M}, w)\cap (BIL(\Sigma),BIL(\Sigma)) = Th_{BIL}(\mathcal{M}\upharpoonright\Sigma, w)$.	
		\item If $\mathcal{M}
		\preccurlyeq_{BIL} \mathcal{N}$, then $(\mathcal{M}\upharpoonright\Sigma)
		\preccurlyeq_{BIL} (\mathcal{N}\upharpoonright\Sigma)$.
		\item If $\mathcal{M}
		\preccurlyeq_{BIL} \mathcal{N}$, then the identical function $id_W$ is a $BIL$-elementary embedding of $\mathcal{M}$ into $\mathcal{N}$.
		\item If $g:\mathcal{M}\cong\mathcal{N}$, then $g$ is a $BIL$-elementary embedding of $\mathcal{M}$ into $\mathcal{N}$; in particular, for any $w \in W$, we have $Th_{BIL}(\mathcal{M}, w) = Th_{BIL}(\mathcal{N}, g(w))$.
		\item A function $g$ is a
		$BIL$-elementary embedding of $\mathcal{M}$ into $\mathcal{N}$ iff
		there exists a (unique) $\mathcal{N}' \preccurlyeq_{BIL}
		\mathcal{N}$ such that $g: \mathcal{M} \cong \mathcal{N}'$. We will denote $\mathcal{N}'$ by $g(\mathcal{M})$.
		\item If $g$ is a
		$BIL$-elementary embedding of $\mathcal{M}$ into $\mathcal{N}$, then
		there exists a $\Theta$-model $\mathcal{M}'$ and function $g'$ such that we have: (a) $\mathcal{M} \preccurlyeq_{BIL} \mathcal{M}'$, (b) $g \subseteq g'$, and (c) $g':\mathcal{M}' \cong \mathcal{N}$.
	\end{enumerate}
\end{lemma}
\begin{proof}
	Part 1 is immediate from the definition, as for Part 2, it is clear that if $\mathcal{M} \subseteq \mathcal{N}$, then also $(\mathcal{M}\upharpoonright\Sigma)
	\subseteq (\mathcal{N}\upharpoonright\Sigma)$. It remains to show the coincidence of theories in the reducts of $\mathcal{M}$ and $\mathcal{N}$. We reason as follows for any given $w \in W$:
	\begin{align*}
	Th_{BIL}(\mathcal{M}\upharpoonright\Sigma, w) &= Th_{BIL}(\mathcal{M}, w) \cap(BIL(\Sigma),BIL(\Sigma))&&\text{(by Part 1)}\\
	&= Th_{BIL}(\mathcal{N}, w) \cap (BIL(\Sigma),BIL(\Sigma)) &&\text{(by $\mathcal{M}
		\preccurlyeq_{BIL} \mathcal{N}$)}\\
	&= Th_{BIL}(\mathcal{N}\upharpoonright\Sigma, w)&&\text{(by Part 1)}
	\end{align*}
	Thus we get that $(\mathcal{M}\upharpoonright\Sigma)
	\preccurlyeq_{BIL} (\mathcal{N}\upharpoonright\Sigma)$.	
	
	Parts 3--5 are straightforward. As for Part 6, we reason as  in the classical case, that is to say, we
	choose a set $C$ outside $W \cup U$ such that there exists a bijection
	$g'':C \to (U \setminus h(W))$. We then set $g': = g \cup g''$ and define $\mathcal{M}'$ on the basis of $W': = W
	\cup C$, setting
	$u\mathrel{\prec'}u' :\Leftrightarrow g'(u)\mathrel{\lhd}g'(u')$
	and $u \in V'(p) :\Leftrightarrow g'(u) \in Y(p)$ for all $p \in \Theta$ and $u, u' \in W'$.
\end{proof}
We are now in a position to define the notion of a bi-intuitionistic type:
\begin{definition}\label{D:types}
	Let $(\mathcal{M}, w) \in Pmod_\Theta$, let
	$\Gamma, \Delta\subseteq BIL(\Theta)$. Then we say that:
	\begin{itemize}
		\item $(\Gamma, \Delta)$ is a successor $BIL$-type of $(\mathcal{M},w)$ iff:
		$$
		(\forall \Gamma' \Subset \Gamma)(\forall \Delta' \Subset
		\Delta)(\exists v \succ
		w)((\Gamma',
		\Delta')\subseteq Th_{BIL}(\mathcal{M}, v)).
		$$
		\item $(\Gamma, \Delta)$ is a predecessor $BIL$-type of $(\mathcal{M},w)$ iff:
		$$
		(\forall \Gamma' \Subset \Gamma)(\forall \Delta' \Subset
		\Delta)(\exists v \prec
		w)((\Gamma',
		\Delta')\subseteq Th_{BIL}(\mathcal{M}, v)).
		$$
	\end{itemize}
\end{definition}
It is easy to see that, due to the reflexivity of the accessibility relation $Th_{BIL}(\mathcal{M}, w)$ is always both a successor and a predecessor type of $(\mathcal{M}, w)$. More generally, we will say that a pair of
sets of formulas is a $BIL$-type of
$(\mathcal{M}, w)$ iff it is either a successor or a predecessor type of $(\mathcal{M}, w)$;
we will say that it is a $BIL$-type of $\mathcal{M}$
iff it is a type of $(\mathcal{M}, w)$ for some $w \in W$.

Another straightforward observation is that the bi-intuitionistic
types, as given by Definition \ref{D:types}, are related to
certain classes of bi-intuitionistic formulas. We state this
observation as a corollary (omitting the obvious proof):
\begin{corollary}\label{C:types-formulas}
	Let $(\mathcal{M}, w) \in Pmod_\Theta$, let
	$\Gamma, \Delta\subseteq BIL(\Theta)$. Then all of the following statements hold:
	\begin{enumerate}
		\item $(\Gamma, \Delta)$ is a successor $BIL$-type of $(\mathcal{M},
		w)$ iff:
		$$
		(\forall \Gamma' \Subset \Gamma)(\forall \Delta' \Subset
		\Delta)(\mathcal{M}, w
		\not\models_{BIL} \bigwedge\Gamma'\to
		\bigvee\Delta').
		$$
			\item $(\Gamma, \Delta)$ is a predecessor $BIL$-type of $(\mathcal{M},
		w)$ iff:
		$$
		(\forall \Gamma' \Subset \Gamma)(\forall \Delta' \Subset
		\Delta)(\mathcal{M}, w
		\models_{BIL} \bigwedge\Gamma'\ll
		\bigvee\Delta').
		$$
	\end{enumerate}
\end{corollary}
We now define what it means for a bi-intuitionistic type to be realized:
\begin{definition}\label{D:types-realization}
	Let $(\mathcal{M}, w) \in Pmod_\Theta$, let $\mathcal{M}
	\preccurlyeq_{BIL} \mathcal{N}$, let
	$\Gamma, \Delta\subseteq BIL(\Theta)$. Then we say that:
	\begin{itemize}
		\item If $(\Gamma, \Delta)$ is a successor $BIL$-type of $(\mathcal{M},
		w)$, then $(\Gamma, \Delta)$ is realized in $\mathcal{N}$ iff there exists a $v \in U$ such that
		$w \lhd v$ and we have $(\Gamma,
		\Delta)\subseteq Th_{BIL}(\mathcal{N}, v)$.
		
		\item If $(\Gamma, \Delta)$ is a predecessor $BIL$-type of $(\mathcal{M},
		w)$, then $(\Gamma, \Delta)$ is realized in $\mathcal{N}$ iff there exists a $v \in U$ such that
		$v \lhd w$ and we have $(\Gamma,
		\Delta)\subseteq Th_{BIL}(\mathcal{N}, v)$.
	\end{itemize}
\end{definition}
For an intuitionistic model to realize all of its bi-intuitionistic types
in the sense of Definition \ref{D:types-realization} is a 
rare property that merits a special name:
\begin{definition}\label{D:saturation}
	Let $\mathcal{M}$ be a $\Theta$-model. We say that $\mathcal{M}$
	is $BIL$-saturated iff it realizes all $BIL$-types
	of $\mathcal{M}$.
\end{definition}
Another interesting situation occurs when a model happens to realize every type of its proper $BIL$-elementary submodel. We collect some properties of type realization in the following lemma:
\begin{lemma}\label{L:type-realization}
	Let $\mathcal{M}'
	\preccurlyeq_{BIL}\mathcal{M}
	\preccurlyeq_{BIL} \mathcal{N}$ be a $BIL$-elementary chain of $\Theta$-models, let $\mathcal{N}'$ be a  $\Theta$-model such that for some functions $g$ we have $g:\mathcal{N}\cong \mathcal{N}'$. Finally, let $\Sigma \subseteq \Theta$. Then the following statements hold:
	\begin{enumerate}
		\item Every $BIL$-type of $\mathcal{M}'$, that is realized in $\mathcal{M}$, is also realized in $\mathcal{N}$.
		
		\item If $\mathcal{N}$ realizes every $BIL$-type of $\mathcal{M}$, then $\mathcal{N}\upharpoonright\Sigma$ realizes every $BIL$-type of $\mathcal{M}\upharpoonright\Sigma$.
		
		\item For every $v \in U$, $(\Gamma, \Delta)$ is a successor (resp. predecessor) $BIL$-type of $(\mathcal{N}, v)$ iff it is a successor (resp. predecessor) $BIL$-type of $(\mathcal{N}', g(v))$.    
	\end{enumerate} 	
\end{lemma}
\begin{proof}
	(Part 1). Assume that $(\Gamma, \Delta)$ is a successor $BIL$-type of $(\mathcal{M}',w)$ that is realized in $\mathcal{M}$. This means that for an appropriate $v \in W$ such that $w\mathrel{\prec}v$ we have $(\Gamma,
	\Delta)\subseteq Th_{BIL}(\mathcal{M}, v) = Th_{BIL}(\mathcal{N}, v)$. Note that, by $\mathcal{M}
	\preccurlyeq_{BIL} \mathcal{N}$, we also have $w\mathrel{\lhd}v$, whence $(\Gamma, \Delta)$ is also realized in $\mathcal{N}$. We argue similarly in case when $(\Gamma, \Delta)$ is a predecessor $BIL$-type.
	
	(Part 2). Assume that $w \in W$, and that $(\Gamma, \Delta)$ is a successor (resp. predecessor) $BIL$-type of $(\mathcal{M}\upharpoonright\Sigma, w)$. Then it follows from the respective part of Corollary \ref{C:types-formulas}, that $(\Gamma, \Delta)$ is also a successor (resp. predecessor) $BIL$-type of $(\mathcal{M}, w)$. But then $(\Gamma, \Delta)$ must be realized in $\mathcal{N}$, which means, by definition, that for an appropriate  $v \in U$ such that $w\mathrel{\lhd}v$ (resp. $v\mathrel{\lhd}w$), we have that $(\Gamma, \Delta)\subseteq Th_{BIL}(\mathcal{N}, v)$. But this means, further, that we also have:
	\begin{align*}
	(\Gamma, \Delta) &= (\Gamma\cap BIL(\Sigma), \Delta\cap  BIL(\Sigma)) = (\Gamma, \Delta)\cap  BIL(\Sigma)\\ 
	&\subseteq Th_{BIL}(\mathcal{N}, v)\cap  (BIL(\Sigma),BIL(\Sigma)) = Th_{BIL}(\mathcal{N}\upharpoonright\Sigma, v),
	\end{align*}
	where the last equality holds by Lemma \ref{L:embedding}.1. This means that $(\Gamma, \Delta)$ is also realized in $\mathcal{N}\upharpoonright\Sigma$.
	
	(Part 3). By Corollary \ref{C:types-formulas} and Lemma \ref{L:embedding}.4.   
\end{proof}

\section{Bi-asimulations and Bi-unravellings}\label{S:unravel}
We devote this subsection to adaptations of some classic constructions in the model theory of classical modal logic to the case of bi-intuitionistic logic. The first one is the bi-intuitionistic analogue of bisimulation relation, originally introduced in \cite{ba1}:
\begin{definition}\label{D:asimulation}
{\em Let $(\mathcal{M}_1, w_1)$, $(\mathcal{M}_2, w_2)\in Pmod_\Theta$. A binary relation $A$ is called a
\emph{bi-asimulation from $(\mathcal{M}_1,w_1)$ to
$(\mathcal{M}_2,w_2)$} iff for any $i,j$ such that $\{ i,j \} = \{
1, 2 \}$, any $v \in W_i$, $s,t \in W_j$, any propositional letter
$p \in \Theta$ the following conditions hold:
\begin{align}
&A \subseteq (W_1 \times W_2) \cup (W_2\times
W_1)\label{E:c22}\tag{\text{w-type}}\\
&w_1\mathrel{A}w_2\label{E:c11}\tag{\text{elem}}\\
&(v\mathrel{A}s \wedge  v \in V_i(p)) \Rightarrow s \in V_j(p)))\label{E:c33}\tag{\text{s-atom}}\\
&(v\mathrel{A}s \wedge s\mathrel{R_j}t) \Rightarrow \exists u \in
W_i(v\mathrel{R_i}u \wedge t\mathrel{A}u \wedge
u\mathrel{A}t)\label{E:c44}\tag{\text{s-back}}\\
&(v\mathrel{A}s \wedge u\mathrel{R_i}v) \Rightarrow \exists t \in
W_j(t\mathrel{R_j}s \wedge t\mathrel{A}u \wedge
u\mathrel{A}t)\label{E:c55}\tag{\text{s-forth}}
\end{align}
}
\end{definition}
Bi-intuitionistic propositional formulas are known to be preserved
under bi-asimulations. More precisely, if $(\mathcal{M}_1, w_1)$, and
$(\mathcal{M}_2, w_2)$ are pointed $\Theta$-models and $A$ is an
bi-asimulation from $(\mathcal{M}_1,w_1)$ to $(\mathcal{M}_2,w_2)$,
then $Th^+_{BIL}(\mathcal{M}_1, w_1) \subseteq
Th^+_{BIL}(\mathcal{M}_2, w_2)$. Moreover, preservation under
bi-asimulations is known to semantically characterize $BIL$ as a
fragment of classical first-order logic, see \cite{ba1,o3} for 
proofs.

Next we consider bi-asimulations between $BIL$-saturated
models. The following lemma states that such bi-asimulations can be
defined in an easy and natural way:
\begin{lemma}\label{L:asimulations}
	Let $(\mathcal{M}_1, w_1)$, $(\mathcal{M}_2, w_2)\in Pmod_\Theta$. If \[
	Th^+_{BIL}(\mathcal{M}_1, w_1) \subseteq
	Th^+_{BIL}(\mathcal{M}_2, w_2)
	\]
	and both $\mathcal{M}_1$ and
	$\mathcal{M}_2$ are intuitionistically saturated, then the
	relation $A$ such that for all $u \in W_i$, $s \in W_j$ if $\{ i,j
	\} = \{ 1,2 \}$, then
	$$
	u\mathrel{A}s \Leftrightarrow (Th^+_{BIL}(\mathcal{M}_i, u)
	\subseteq Th^+_{BIL}(\mathcal{M}_j, s))
	$$
	is a bi-asimulation from $(\mathcal{M}_1, w_1)$ to $(\mathcal{M}_2,
	w_2)$.
\end{lemma}
\begin{proof}
	The relation $A$, as defined in the lemma, obviously satisfies
	conditions \eqref{E:c22}, \eqref{E:c11}, and \eqref{E:c33} given
	in Definition \ref{D:asimulation}. We check the remaining
	conditions.
	
	Condition \eqref{E:c44}. Assume that $u\mathrel{A}s$, so that $Th^+_{BIL}(\mathcal{M}_i, u)
	\subseteq Th^+_{BIL}(\mathcal{M}_j, s)$, and that for some $t \in
	W_j$ we have $s\mathrel{R_j}t$. Let $(\Gamma', \Delta') \subseteq
	Th_{BIL}(\mathcal{M}_j, t)$ be finite. Then, of course, we have:
	$$
	\mathcal{M}_j, s \not\models_{BIL} \bigwedge\Gamma' \to
	\bigvee\Delta',
	$$
	and, by $u\mathrel{A}s$:
	$$
	\mathcal{M}_i, u \not\models_{BIL} \bigwedge\Gamma' \to
	\bigvee\Delta'.
	$$
	The latter means that $(\Gamma', \Delta') $ must be $BIL$-satisfied
	by some $R_i$-successor of $u$. Therefore, by bi-intuitionistic
	saturation of both $\mathcal{M}_1$ and $\mathcal{M}_2$, we get
	that for some $w \in W_i$ such that $u\mathrel{R_i}w$ the theory
	$Th_{BIL}(\mathcal{M}_j, t)$ is $BIL$-satisfied at $(\mathcal{M}_i,
	w)$. It follows immediately that
	$$
	Th_{BIL}(\mathcal{M}_i, w) = Th_{BIL}(\mathcal{M}_j, t),
	$$
	whence
	$$
	Th^+_{BIL}(\mathcal{M}_i, w) = Th^+_{BIL}(\mathcal{M}_j, t),
	$$
	which, in turn, means that both $w\mathrel{A}t$ and
	$t\mathrel{A}w$.
	
		Condition \eqref{E:c55}. Assume that $u\mathrel{A}s$, so that $Th^+_{BIL}(\mathcal{M}_i, u)
	\subseteq Th^+_{BIL}(\mathcal{M}_j, s)$, and that for some $w \in
	W_i$ we have $w\mathrel{R_i}u$. Let $(\Gamma', \Delta') \subseteq
	Th_{BIL}(\mathcal{M}_i, w)$ be finite. Then, of course, we have:
	$$
	\mathcal{M}_i, u \models_{BIL} \bigwedge\Gamma' \ll
	\bigvee\Delta',
	$$
	and, by $u\mathrel{A}s$:
	$$
	\mathcal{M}_j, s \models_{BIL} \bigwedge\Gamma' \ll
	\bigvee\Delta'.
	$$
	The latter means that $(\Gamma', \Delta') $ must be $BIL$-satisfied
	by some $R_j$-predecessor of $s$. Therefore, by bi-intuitionistic
	saturation of both $\mathcal{M}_1$ and $\mathcal{M}_2$, we get
	that for some $t \in W_j$ such that $t\mathrel{R_j}s$ the theory
	$Th_{BIL}(\mathcal{M}_i, w)$ is $BIL$-satisfied at $(\mathcal{M}_j,
	t)$. It follows immediately that
	$$
	Th_{BIL}(\mathcal{M}_i, w) = Th_{BIL}(\mathcal{M}_j, t),
	$$
	whence
	$$
	Th^+_{BIL}(\mathcal{M}_i, w) = Th^+_{BIL}(\mathcal{M}_j, t),
	$$
	which, in turn, means that both $w\mathrel{A}t$ and
	$t\mathrel{A}w$.
\end{proof}

The second idea is a bi-intuitionistic adaptation of unravelling, a well-known technique also applicable, as seen in \cite{chagrov}, to intuitionistic propositional logic. For the setting of $BIL$ we modify this idea to obtain the notion of bi-unravelling.

For a given
$(\mathcal{M}, w)\in Pmod_\Theta$, the model $\mathcal{M}^{un}_w = \langle
W^{un}_w, \prec^{un}_w, V^{un}_w\rangle$, called the
bi-unravelling of $\mathcal{M}$ around $w$ is defined as follows.
\begin{itemize}
\item $W^{un}_w = \{ \bar{u}_n \in W^n\mid u_1 = w,\,(\forall i <
n)(u_i\mathrel{(\prec \cup \succ)}u_{i + 1}\,\&\,u_i \neq u_{i + 1}) \}$;
in other words, $W^{un}_w$ is the set of finite $(\prec \cup \succ)$-chains starting at $w$ where no two adjacent nodes coincide.

\item $\prec^{un}_w$ is the reflexive and transitive closure of the
following relation:
\begin{align*}
\rho^{un}_w = \{ (s,t)\in W^{un}_w \times W^{un}_w \mid
(\exists n < \omega \exists \bar{w}_{n + 1} \in W^{n+1})((s = &\bar{w}_n\,\&\,t = \bar{w}_{n + 1}\,\&\,w_n\mathrel{\prec} w_{n + 1})\vee\\ 
&(s = \bar{w}_{n + 1}\,\&\,t = \bar{w}_{n}\,\&\,w_n\mathrel{\succ}w_{n + 1})) \};
\end{align*}
\item For arbitrary $p \in \Theta$, we have $V^{un}_w(p) = \{
\bar{u}_n \in W^{un}_w \mid u_n \in V(p) \}$.
\end{itemize}
The structure of bi-unravellings is a bit more complicated than the structure of simple unravellings applied in classical modal logic and intuitionistic logic. In particular, bi-unravellings do not typically lead to tree-like frames. Therefore, before we relate bi-unravellings to bi-asimulations and bi-intuitionistic theories, we need to look somewhat more closely into the structure of bi-unravelled models. We do this by proving a couple of technical lemmas.
\begin{lemma}\label{L:rho-relation}
Let $(\mathcal{M}, w)\in Pmod_\Theta$, let $k > 0$, let the tuple $\bar{\alpha}_k \in (W^{un}_w)^k$ consist of pairwise distinct tuples such that for all $i < k$ we have $\alpha_i\mathrel{\rho^{un}_w}\alpha_{i + 1}$. Then there exist $m, n \geq 0$ and (not necessarily pairwise distinct) $\bar{v}_m \in W^m$ and $\bar{u}_n \in W^n$ such that: (1) $v_m\mathrel{\prec}\ldots\mathrel{\prec}v_1\mathrel{\prec} end(\alpha_{m + 1})\mathrel{\prec}u_1\mathrel{\prec}\ldots\mathrel{\prec}u_n$, (2) $k = m + n + 1$, (3) for all $1 \leq i \leq m$, $\alpha_i = (\alpha_{m + 1})^\frown\bar{v}_{m - i + 1}$, and (4) for all $1 \leq i \leq n$, $\alpha_{m + 1 + i} = (\alpha_{m + 1})^\frown\bar{u}_{i}$. In other words, $\bar{\alpha}_k$ is representable in the form:
\begin{align*}
	\alpha_1 &= (\alpha_{m + 1})^\frown(\bar{v}_{m})\\
	&\vdots\\
	\alpha_m &= (\alpha_{m + 1})^\frown v_1\\
	&\alpha_{m + 1}\\
	\alpha_{m + 2} &= (\alpha_{m + 1})^\frown u_1\\
	&\vdots\\
	\alpha_k &= \alpha_{m + n + 1} = (\alpha_{m + 1})^\frown(\bar{u}_{n}).
\end{align*}	
\end{lemma}
\begin{proof}
We proceed by induction on $k > 0$.

\textit{Induction Basis}. $k = 1$. Then we set $m = n := 0$, $\bar{v}_m =\bar{u}_n := \Lambda$, and get the Lemma trivially satisfied.

\textit{Induction Step}. $k = l + 1$ for some $l \geq 1$. Then compare $\alpha_1$ and $\alpha_2$. Two cases are possible:

\textit{Case 1}. There exist $\bar{w}_{r + 1} \in W^{r+1}$ such that $\alpha_1 = \bar{w}_{r + 1}$, $\alpha_2 = \bar{w}_{r}$, and $w_r\mathrel{\succ} w_{r + 1}$. Then we apply Induction Hypothesis to $(\alpha_2,\ldots, \alpha_k)$ and find $m_1, n_1 \geq 0$ and (not necessarily pairwise distinct) $\bar{v}'_{m_1} \in W^{m_1}$ and $\bar{u}'_{n_1} \in W^{n_1}$ such that: (1) $v'_{m_1}\mathrel{\prec}\ldots\mathrel{\prec}v'_1\mathrel{\prec} end(\alpha_{m'_1 + 2})\mathrel{\prec}u'_1\mathrel{\prec}\ldots\mathrel{\prec}u'_{n_1}$, (2) $k - 1 = m_1 + n_1 + 1$, (3) for all $1 \leq i \leq m_1$, $\alpha_{i + 1} = (\alpha_{m_1 + 2})^\frown\bar{v}'_{m_1 - i + 1}$, and (4) for all $1 \leq i \leq n_1$, $\alpha_{m_1 + 2 + i} = (\alpha_{m_1 + 2})^\frown\bar{u}'_{i}$. In other words, $(\alpha_2,\ldots, \alpha_k)$ is representable in the form:
\begin{align*}
	\alpha_2 &= (\alpha_{m_1 + 2})^\frown\bar{v}'_{m_1}\\
	&\vdots\\
	\alpha_{m_1 + 1} &= (\alpha_{m_1 + 2})^\frown v'_1\\
	&\alpha_{m_1 + 2}\\
	\alpha_{m_1 + 2} &= (\alpha_{m_1 + 2})^\frown u'_1\\
	&\vdots\\
	\alpha_k &= \alpha_{m_1 + n_1 + 2} = (\alpha_{m_1 + 2})^\frown\bar{u}'_{n_1}
\end{align*}
But then we can set $m: = m_1 + 1$, $n : = n_1$, $\bar{v}_m  := (\bar{v}'_{m_1})^\frown(w_{r + 1})$, $\bar{u}_{n}:= \bar{u}'_{n_1}$, and get the Lemma verified.

\textit{Case 2}. There exist $\bar{w}_{r + 1} \in W^{r+1}$ such that $\alpha_1 = \bar{w}_{r}$, $\alpha_2 = \bar{w}_{r + 1}$, and $w_r \prec w_{r + 1}$. Again we apply Induction Hypothesis to $(\alpha_2,\ldots, \alpha_k)$ and find $m_1, n_1 \geq 0$ and (not necessarily pairwise distinct) $\bar{v}'_{m_1} \in W^{m_1}$ and $\bar{u}'_{n_1} \in W^{n_1}$  allowing to represent  $(\alpha_2,\ldots, \alpha_k)$ in the form given in Case 1. But then, we must have $(\alpha_{m_1 + 2})^\frown\bar{v}'_{m_1} = \alpha_2 = \bar{w}_{r + 1}$. Now, assume that $m_1 > 0$; then we would have $v'_{m_1} = w_{r + 1}$, and also $\alpha_3 = (\alpha_{m_1 + 2})^\frown\bar{v}'_{m_1 - 1} = \bar{w}_{r} = \alpha_1$, but the latter is in contradiction with our assumption that $\bar{\alpha}_k \in (W^{un}_w)^k$ consists of pairwise distinct tuples. Therefore, we must have $m_1 = 0$, so that  $(\alpha_2,\ldots, \alpha_k)$ in fact must have the form:
\begin{align*}
	\alpha_2 &= \alpha_{m_1 + 2} = \bar{w}_{r + 1}\\
	\alpha_3 &=\alpha_{m_1 + 2} = (\alpha_{m_1 + 2})^\frown u'_1\\
	&\vdots\\
	\alpha_k &= \alpha_{m_1 + n_1 + 2} = (\alpha_{m_1 + 2})^\frown\bar{u}'_{n_1}
\end{align*}
 But then we can set $m: = 0$, $n : = n_1 + 1$, $\bar{v}_m  := \Lambda$, $\bar{u}_{n}:=  (w_{r + 1})^\frown\bar{u}'_{n_1}$; these settings clearly satisfy the Lemma. 
\end{proof}
\begin{lemma}\label{L:prec-relation}
	Let $(\mathcal{M}, w)\in Pmod_\Theta$, let $\alpha, \beta \in (W^{un}_w)$ be such that $\alpha\mathrel{\prec^{un}_w}\beta$. Then there exist $\gamma \in W^{un}_w$, $m, n \geq 0$ and (not necessarily pairwise distinct) $\bar{v}_m \in W^m$ and $\bar{u}_n \in W^n$ such that $v_{m}\mathrel{\prec}\ldots\mathrel{\prec}v_1\mathrel{\prec} end(\gamma)\mathrel{\prec}u_1\mathrel{\prec}\ldots\mathrel{\prec}u_{n}$, and:
	\begin{align*}
		\alpha &= \gamma^\frown\bar{v}_{m}\\
		\beta &=  \gamma^\frown\bar{u}_{n}.
	\end{align*}	
\end{lemma}
\begin{proof}
	Since $\prec^{un}_w$ is just the reflexive and transitive closure of $\rho^{un}_w$, there must be an $l \geq 1$ and $\bar{\alpha}'_l \in (W^{un}_w)^l$ such that $\alpha'_1\mathrel{\rho^{un}_w}\ldots\mathrel{\rho^{un}_w}\alpha'_l$, $\alpha = \alpha'_1$ and $\beta = \alpha'_l$. In case $\bar{\alpha}_l$ consists of tuples that are not pairwise distinct, we can just replace in $\bar{\alpha}'_l$ every sub-sequence of the form $\beta^\frown(\bar{\delta}_r)^\frown\beta$ with $\beta$. The result of these replacements will be some $\bar{\alpha}_k \in (W^{un}_w)^k$ consisting of pairwise distinct tuples, such that $k > 0$, $\alpha_1\mathrel{\rho^{un}_w}\ldots\mathrel{\rho^{un}_w}\alpha_k$, $\alpha = \alpha_1$ and $\beta = \alpha_k$. But then we can apply Lemma \ref{L:rho-relation}, to find $m, n \geq 0$ and (not necessarily pairwise distinct) $\bar{v}_m \in W^m$ and $\bar{u}_n \in W^n$ such that $v_{m}\mathrel{\prec}\ldots\mathrel{\prec}v_1\mathrel{\prec} end(\alpha_{m + 1})\mathrel{\prec}u_1\mathrel{\prec}\ldots\mathrel{\prec}u_{n}$, and we have:
	\begin{align*}
		\alpha_1 &= (\alpha_{m + 1})^\frown(\bar{v}_{m})\\
		&\vdots\\
		\alpha_m &= (\alpha_{m + 1})^\frown v_1\\
		&\alpha_{m + 1}\\
		\alpha_{m + 2} &= (\alpha_{m + 1})^\frown u_1\\
		&\vdots\\
		\alpha_k &= \alpha_{m + n + 1} = (\alpha_{m + 1})^\frown(\bar{u}_{n}).
	\end{align*}
Setting then $\gamma:= \alpha_{m + 1}$, we get our Lemma verified.
\end{proof}

The following two lemmas sum up the basic facts about bi-unravellings:
\begin{lemma}\label{L:unravelling-mod}
	Let $(\mathcal{M}, w) \in Pmod_\Theta$. Then $(\mathcal{M}^{un}_w, w)\in Pmod_\Theta$.
\end{lemma}
\begin{proof}
By definition, $\prec^{un}_w$ is reflexive and transitive. As for antisymmetry, assume that, for a given
$\bar{w}_r, \bar{v}_n \in W^{un}_w$, we have
$$
\bar{w}_r\mathrel{\prec^{un}_w}\bar{v}_n \,\&\,
\bar{v}_n\mathrel{\prec^{un}_w}\bar{w}_r.
$$
By Lemma \ref{L:prec-relation}, there must be, on the one hand, some $m_0, n_0 \geq 0$ and $\gamma \in W^{un}_w$  and (not necessarily pairwise distinct) $\bar{v}_{m_0} \in W^{m_0}$ and $\bar{u}_{n_0} \in W^{n_0}$ such that $v_{m_0}\mathrel{\prec}\ldots\mathrel{\prec}v_1\mathrel{\prec}u_1\mathrel{\prec}\ldots\mathrel{\prec}u_{n_0}$, and:
\begin{align*}
	\alpha &= \gamma^\frown\bar{v}_{m_0}\\
	\beta &=  \gamma^\frown\bar{u}_{n_0}.
\end{align*}
On the other hand, again by Lemma \ref{L:prec-relation}, there must be some $m_1, n_1 \geq 0$ and $\delta \in W^{un}_w$ and (not necessarily pairwise distinct) $\bar{s}_{m_1} \in W^{m_1}$ and $\bar{t}_{n_1} \in W^{n_1}$ such that $s_{m_1}\mathrel{\prec}\ldots\mathrel{\prec}s_1\mathrel{\prec}t_1\mathrel{\prec}\ldots\mathrel{\prec}t_{n_1}$, and:
\begin{align*}
	\beta &= \delta^\frown\bar{s}_{m_1}\\
	\alpha &=  \delta^\frown\bar{t}_{n_1}.
\end{align*} 	
We then compare $l_\gamma$ and $l_\delta$, the lengths of $\gamma$ and $\delta$, respectively. Assume, w.l.o.g., that $l_\gamma \geq l_\delta$. We must have then $\gamma^\frown\bar{v}_{m_0} = \alpha = \delta^\frown\bar{t}_{n_1}$, whence it follows that $\gamma = \delta^\frown\bar{t}_{n_1 - m_0}$ and so also $\bar{v}_{m_0} = (t_{n_1 - m_0 + 1},\ldots, t_{n_1})$. But then, notice that we have both $v_{m_0}\mathrel{\prec}\ldots\mathrel{\prec}v_1$ and $t_{n_1 - m_0 + 1}\mathrel{\prec}\ldots\mathrel{\prec}t_{n_1}$. It follows then, by the antisymmetry of $\prec$, that we have:
$$
v_{m_0} = t_{n_1 - m_0 + 1} = \ldots = v_1 = t_{n_1},
$$
so, in particular, we get that $v_1 =\ldots = v_{m_0}$. Next, note we have $\alpha = \gamma^\frown\bar{v}_{m_0} \in W^{un}_w$, which means that no two adjacent elements in $\bar{v}_{m_0}$ may coincide. Hence we must have $m_0 = 1$, and thus also $\alpha = \gamma^\frown v_1$.

By a parallel argument, we get that, on the other hand, we must have $\gamma = \delta^\frown\bar{s}_{m_1 - n_0}$ and also $\bar{u}_{n_0} = (s_{m_1 - n_0 + 1},\ldots, s_{m_1})$. But then, notice that we have both $u_{1}\mathrel{\prec}\ldots\mathrel{\prec}u_{n_0}$ and $s_{m_1}\mathrel{\prec}\ldots\mathrel{\prec}s_{m_1 - n_0 + 1}$, whence it follows, by the antisymmetry of $\prec$, that we have:
$$
u_{1} = s_{m_1} = \ldots = u_{n_0} = s_{m_1 - n_0 + 1},
$$
so, in particular, we get that $u_1 =\ldots = u_{n_0}$. Next, note we have $\beta = \gamma^\frown\bar{u}_{n_0} \in W^{un}_w$, which means that no two adjacent elements in $\bar{v}_{m_0}$ may coincide. Hence we must have $n_0 = 1$, and thus also $\beta = \gamma^\frown u_1$. 

Furthermore, note that, by the choice of $v_1$ and $u_1$, we have $v_1\mathrel{\prec}u_1$. On the other hand, the above chains of equalities, together with the transitivity of $\prec$, also imply that $u_1 = s_{m_1 - n_0 + 1}\mathrel{\prec}t_{n_1} = v_1$, hence also $u_1\mathrel{\prec}v_1$. It follows, next, by the antisymmetry of $\prec$, that $u_1 = v_1$, so that we get that:
$$
\alpha = \gamma^\frown v_1 = \gamma^\frown u_1 = \beta.
$$

To show monotonicity of $V^{un}_w$ w.r.t. $\prec^{un}_w$, assume that
for some $\alpha, \beta \in W^{un}_w$ we have
$\alpha\mathrel{\prec^{un}_w}\beta$. Then we choose, using Lemma \ref{L:prec-relation}, $\gamma \in W^{un}_w$, $m, n \geq 0$ and (not necessarily pairwise distinct) $\bar{v}_m \in W^m$ and $\bar{u}_n \in W^n$ such that $v_{m}\mathrel{\prec}\ldots\mathrel{\prec}v_1\mathrel{\prec}u_1\mathrel{\prec}\ldots\mathrel{\prec}u_{n}$, and:
\begin{align*}
	\alpha &= \gamma^\frown\bar{v}_{m}\\
	\beta &=  \gamma^\frown\bar{u}_{n}.
\end{align*}	
But then, let $p \in \Theta$ be chosen arbitrarily. We reason as follows:
\begin{align*}
	\alpha \in V^{un}_w(p) &\Leftrightarrow v_m \in V(p)\\
	&
	\Rightarrow u_n \in V(p) &&\text{(by transitivity of $\prec$ and monotonicity of $V$)}\\
	&\Leftrightarrow \beta \in V^{un}_w(p) 
\end{align*}	
\end{proof}
\begin{lemma}\label{L:unravelling-asim}
Let $(\mathcal{M}, w)$ be a pointed $\Theta$-model. Then:
\begin{enumerate}
\item The relation $B$ defined as follows:
$$
B := \{(w_n, \bar{w}_n), (\bar{w}_n,w_n) \mid \bar{w}_n \in
W^{un}_w\}.
$$
is a bi-asimulation from $(\mathcal{M}^{un}_w, w)$ to $(\mathcal{M},
w)$ and from $(\mathcal{M},
w)$ to $(\mathcal{M}^{un}_w, w)$;

\item $Th_{BIL}(\mathcal{M}, v_k) = Th_{BIL}(\mathcal{M},
\bar{v}_k)$ for any $\bar{v}_k \in W^{un}_w$;

\item $Th_{BIL}(\mathcal{M}, \bar{u}_n) = Th_{BIL}(\mathcal{M},
\bar{v}_k)$ for any $\bar{u}_n, \bar{v}_k \in W^{un}_w$ whenever
$u_n = v_k$.
\end{enumerate}
\end{lemma}
\begin{proof} (Part 1) It is clear that $B$ satisfies conditions \eqref{E:c11}, \eqref{E:c22}, and \eqref{E:c33}; the latter two conditions are even satisfied both from $(\mathcal{M}^{un}_w, w)$ to $(\mathcal{M},
	w)$ and from $(\mathcal{M},
	w)$ to $(\mathcal{M}^{un}_w, w)$. We consider the remaining conditions:
	
\textit{Condition \eqref{E:c44}}. If $\bar{w}_n \in W^{un}_w$ and $w_n\mathrel{\prec}v$, then two cases are possible:

\textit{Case 1}. $w_n = v$. Then we note that we have $\bar{w}_n\mathrel{\prec^{un}_w}\bar{w}_n$, $v\mathrel{B}\bar{w}_n$ and $\bar{w}_n\mathrel{B}v$.

\textit{Case 2}. $w_n \neq v$. Then we must have $(\bar{w}_n)^\frown v \in W^{un}_w$, $\bar{w}_n\mathrel{\prec^{un}_w}(\bar{w}_n)^\frown v$, $v\mathrel{B}(\bar{w}_n)^\frown v$, and $(\bar{w}_n)^\frown v\mathrel{B}v$.

On the other hand, if $\bar{w}_n\mathrel{\prec^{un}_w}\alpha$ for some $\alpha \in W^{un}_w$ then it follows, by Lemma \ref{L:prec-relation}, that we can choose  some $m_0, n_0 \geq 0$ and $\gamma \in W^{un}_w$  and (not necessarily pairwise distinct) $\bar{v}_{m_0} \in W^{m_0}$ and $\bar{u}_{n_0} \in W^{n_0}$ such that $v_{m_0}\mathrel{\prec}\ldots\mathrel{\prec}v_1\mathrel{\prec} end(\gamma)\mathrel{\prec}u_1\mathrel{\prec}\ldots\mathrel{\prec}u_{n_0}$, and:
\begin{align*}
	\bar{w}_n &= \gamma^\frown\bar{v}_{m_0}\\
	\alpha &=  \gamma^\frown\bar{u}_{n_0}.
\end{align*}  
Again, two cases are possible:

\textit{Case 1}. $m_0 + n_0 = 0$. Then $\bar{w}_n=\alpha$; note, further, that we have then $w_n\mathrel{\prec}w_n$, $w_n\mathrel{B}\alpha$ and $\alpha\mathrel{B}w_n$.

\textit{Case 2}. $m_0 + n_0 > 0$. Then (depending on whether $m_0 = 0$ or $m_0 \neq 0$) we must have that $w_n = end(\gamma)$ or $w_n = v_{m_0}$. In both cases it follows, by transitivity of $\prec$, that we have $w_n \prec u_{n_0}$; on the other hand, we have both $u_{n_0}\mathrel{B}\alpha$ and $\alpha\mathrel{B}u_{n_0}$.

\textit{Condition \eqref{E:c55}}. If $\bar{w}_n \in W^{un}_w$ and $w_n\mathrel{\succ}v$, then two cases are possible:

\textit{Case 1}. $w_n = v$. Then we note that we have $\bar{w}_n\mathrel{\prec^{un}_w}\bar{w}_n$, $v\mathrel{B}\bar{w}_n$ and $\bar{w}_n\mathrel{B}v$.

\textit{Case 2}. $w_n \neq v$. Then we must have $(\bar{w}_n)^\frown v \in W^{un}_w$, $(\bar{w}_n)^\frown v\mathrel{\prec^{un}_w}\bar{w}_n$, $v\mathrel{B}(\bar{w}_n)^\frown v$, and $(\bar{w}_n)^\frown v\mathrel{B}v$.

On the other hand, if $\alpha\mathrel{\prec^{un}_w}\bar{w}_n$ for some $\alpha \in W^{un}_w$ then it follows, by Lemma \ref{L:prec-relation}, that we can choose  some $m_0, n_0 \geq 0$ and $\gamma \in W^{un}_w$ such that and (not necessarily pairwise distinct) $\bar{v}_{m_0} \in W^{m_0}$ and $\bar{u}_{n_0} \in W^{n_0}$ such that $v_{m_0}\mathrel{\prec}\ldots\mathrel{\prec}v_1\mathrel{\prec} end(\gamma)\mathrel{\prec}u_1\mathrel{\prec}\ldots\mathrel{\prec}u_{n_0}$, and:
\begin{align*}
	 \alpha &= \gamma^\frown\bar{v}_{m_0}\\
	\bar{w}_n &=  \gamma^\frown\bar{u}_{n_0}.
\end{align*}  
Again, two cases are possible:

\textit{Case 1}. $m_0 + n_0 = 0$. Then $\bar{w}_n=\alpha$; note, further, that we have then $w_n\mathrel{\prec}w_n$, $w_n\mathrel{B}\alpha$ and $\alpha\mathrel{B}w_n$.

\textit{Case 2}. $m_0 + n_0 > 0$. Then (depending on whether $n_0 = 0$ or $n_0 \neq 0$) we must have that $w_n = end(\gamma)$ or $w_n = u_{n_0}$. In both cases it follows, by transitivity of $\prec$, that we have $w_n \succ v_{m_0}$; on the other hand, we have both $v_{m_0}\mathrel{B}\alpha$ and $\alpha\mathrel{B}v_{m_0}$.

Parts $2$ and $3$ then follow from preservation of bi-intuitionistic formulas under bi-asimulations.
\end{proof}
More generally, we will call an arbitrary $(\mathcal{M}, w) \in Pmod_\Theta$ a bi-unravelled model iff there is an $(\mathcal{N}, v) \in Pmod_\Theta$ and a bijection $g: W \to U^{un(v)}$ satisfying the condition \eqref{E:ic1} of Definition \ref{D:isomorphism} such that $g(v) =w$; or, in other words, if the underlying frame $(W, \prec)$ of $(\mathcal{M}, w)$ is isomorphic to the underlying frame of $(\mathcal{N}^{un(v)}, v)$. It is clear that if $\mathcal{M}$ is a bi-unravelled model, we may always assume that its underlying frame is just a copy of the underlying frame of the respective bi-unravelled model. We will assume, therefore, for the sake of simplicity, that if $(\mathcal{M}, w) \in Pmod_\Theta$ is a bi-unravelled model, then $\mathcal{M}$ is given in the form $\langle W^{un(w)}, \prec^{un(w)}, V\rangle$.

\section{Abstract bi-intuitionistic logics}\label{S:Abstract}
An abstract bi-intuitionistic logic $\mathcal{L}$ is a pair $(L,
\models_\mathcal{L})$, where $L$ maps every signature $\Theta$ to
the set $L(\Theta)$ of $\Theta$-formulas of $\mathcal{L}$ and
$\models_\mathcal{L}$ is a class-relation such that, if
$\alpha\mathrel{\models_\mathcal{L}}\beta$, then there exists a
signature $\Theta$ such that $\alpha \in Pmod_\Theta$,
and $\beta\in L(\Theta)$; informally this is to mean that $\beta$
holds in $\alpha$. The relation $\models_\mathcal{L}$ is only assumed to be defined (i.e. to either hold or fail) for the elements of the class $\bigcup\{Pmod_\Theta\times L(\Theta)\mid \Theta\text{ is a signature}\}$ and to be undefined otherwise.

In order for such a pair $\mathcal{L}$ to count as an abstract
bi-intuitionistic logic, we demand that the following conditions are
satisfied:
\begin{itemize}
\item $\Theta \subseteq \Theta' \Rightarrow L(\Theta) \subseteq
L(\Theta')$.

\item (Isomorphism). If $(\mathcal{M}, w)\in Pmod_\Theta$ and $\mathcal{N} \in Mod_\Theta$,
$\phi \in L(\Theta)$ and $f:\mathcal{M} \cong \mathcal{N}$, then:
$$
\mathcal{M}, w \models_\mathcal{L} \phi \Leftrightarrow
\mathcal{N}, f(w) \models_\mathcal{L} \phi.
$$

\item (Expansion). If $\Theta$ is a signature, $\phi \in
L(\Theta)$, $\Theta \subseteq \Theta'$, and $\mathcal{M}$ is a
$\Theta'$-model, then:
$$
\mathcal{M}, w \models_\mathcal{L} \phi \Leftrightarrow
\mathcal{M}\upharpoonright\Theta, w \models_\mathcal{L} \phi.
$$

\item (Occurrence). If $\phi \in L(\Theta)$ for some signature ,
then there is a finite $\Theta_\phi \subseteq \Theta$ such that
for every $\Theta'$-model $\mathcal{M}$, the relation
$\mathcal{M}\models_\mathcal{L}\phi$ is defined iff $\Theta_\phi
\subseteq \Theta'$.

\item (Closure). For every signature $\Theta$ and all $\phi, \psi
\in L(\Theta)$, we have \[
\bot, \phi \ll \psi, \phi \to \psi, \phi \wedge \psi, \phi
\vee \psi \in L(\Theta),
\]
 that is to say, $\mathcal{L}$ always includes falsum and is closed
under intuitionistic implication, co-implication, conjunction and disjunction.
\end{itemize}

We further define that given a pair of abstract intuitionistic
logics $\mathcal{L}$ and $\mathcal{L}'$, we say that
$\mathcal{L}'$ extends $\mathcal{L}$ and write $\mathcal{L}
\sqsubseteq \mathcal{L}'$ when for all signatures $\Theta$
and $\phi\in L(\Theta)$ there exists a $\psi\in L'(\Theta)$ such
that for arbitrary pointed $\Theta$-model $(\mathcal{M}, w)$ it is
true that:
$$
\mathcal{M}, w\models_\mathcal{L}\phi \Leftrightarrow \mathcal{M},
w\models_{\mathcal{L}'}\psi.
$$
If both $\mathcal{L} \sqsubseteq \mathcal{L}'$ and
$\mathcal{L}' \sqsubseteq \mathcal{L}$ holds, then we say that
the logics $\mathcal{L}$ and $\mathcal{L}'$ are \emph{expressively
equivalent} and write $\mathcal{L} \equiv \mathcal{L}'$.

It is easy to see that bi-intuitionistic propositional logic itself
turns out to be an abstract bi-intuitionistic logic $\mathsf{BIL} =
(BIL, \models_{BIL})$ under this definition. It is also obvious that
the above definitions and conventions about bi-intuitionistic
theories can be carried over to an arbitrary abstract
bi-intuitionistic logic $\mathcal{L}$ replacing everywhere $BIL$ with
$\mathcal{L}$, including such notions as $\mathcal{L}$-elementary submodel,
$\mathcal{L}$-elementary embedding, $\mathcal{L}$-type, $\mathcal{L}$-saturation of a model, etc.

In this paper, our specific interest is in the extensions of
$\mathsf{BIL}$. Since every abstract intuitionistic logic
$\mathcal{L}$ extending $\mathsf{BIL}$ must have an equivalent for
every bi-intuitionistic propositional formula, we will just assume
that for every signature $\Theta$ we have $BIL(\Theta) \subseteq
L(\Theta)$ and that for every $\varphi \in BIL(\Theta)$ and every
pointed $\Theta$-model $(\mathcal{M}, w)$ we have that:
$$
\mathcal{M}, w\models_\mathcal{L}\varphi \Leftrightarrow
\mathcal{M}, w\models_{BIL}\varphi,
$$
so that all the bi-intuitionistic
propositional formulas are present in $\mathcal{L}$ in their usual
form and with their usual meaning, and whatever other formulas
that $\mathcal{L}$ may contain are distinct from the elements of
$BIL(\Theta)$.

We can immediately state the following corollary to Lemma
\ref{L:asimulations} for arbitrary extensions of $\mathsf{BIL}$:

\begin{corollary}\label{L:asimulationscorollary}
Let $\mathsf{BIL} \sqsubseteq \mathcal{L}$, and let
$(\mathcal{M}_1, w_1)$, $(\mathcal{M}_2, w_2)$ be two pointed $\Theta$-models. If
$Th^+_{BIL}(\mathcal{M}_1, w_1) \subseteq Th^+_{BIL}(\mathcal{M}_2,
w_2)$ and both $\mathcal{M}_1$ and $\mathcal{M}_2$ are
$\mathcal{L}$-saturated, then the relation $A$ such that for all
$u \in W_i$, $s \in W_j$ if $\{ i,j \} = \{ 1,2 \}$, then
$$
u\mathrel{A}s \Leftrightarrow (Th^+_{BIL}(\mathcal{M}_i, u)
\subseteq Th^+_{BIL}(\mathcal{M}_j, s))
$$
is a bi-asimulation from $(\mathcal{M}_1, w_1)$ to $(\mathcal{M}_2,
w_2)$.
\end{corollary}
To prove this, we just repeat the proof of Lemma
\ref{L:asimulations} using the fact that every
$\mathcal{L}$-saturated model is of course $BIL$-saturated.

Some of the extensions of $\mathsf{BIL}$ turn out to be better than
others in that they have useful model-theoretic properties. We
define some of the relevant properties below.
\begin{definition}\label{D:properties}
Let $\mathcal{L} = (L, \models_\mathcal{L})$ be an abstract
bi-intuitionistic logic. Then:
\begin{itemize}
\item $\mathcal{L}$ is \textbf{preserved under bi-asimulations}, iff
for every signature $\Theta$ and for arbitrary pointed
$\Theta$-models $(\mathcal{M}_1, w_1)$ and $(\mathcal{M}_2, w_2)$,
whenever $A$ is a bi-asimulation from $(\mathcal{M}_1, w_1)$ to
$(\mathcal{M}_2, w_2)$, then the inclusion
$Th^+_\mathcal{L}(\mathcal{M}_1, w_1) \subseteq
Th^+_\mathcal{L}(\mathcal{M}_2, w_2)$ holds.

\item $\mathcal{L}$ is \textbf{$\star$-compact}, iff an arbitrary
$L(\Theta)$-theory  $(\Gamma, \Delta)$ is
$\mathcal{L}$-satisfiable, whenever for every finite $\Gamma'
\subseteq \Gamma$ and $\Delta' \subseteq \Delta$, the theory
$(\Gamma', \Delta')$ is $\mathcal{L}$-satisfiable.

\item $\mathcal{L}$ has \textbf{Tarski Union Property (TUP)} iff
for every $\mathcal{L}$-elementary chain

$$
\mathcal{M}_0 \preccurlyeq_\mathcal{L},\ldots,
\preccurlyeq_\mathcal{L} \mathcal{M}_n
\preccurlyeq_\mathcal{L},\ldots
$$
it is true that:
$$
\mathcal{M}_n \preccurlyeq_\mathcal{L} \bigcup_{n \in
\omega}\mathcal{M}_n
$$

for all $n \in \omega$.
\end{itemize}
\end{definition}
If an abstract bi-intuitionistic logic is preserved under
bi-asimulations, then every formula of this logic is monotonic w.r.t.
accessibility relation. More precisely, the following lemma holds:
\begin{lemma}\label{L:monotonicity}
If $\mathcal{L}$ is an abstract bi-intuitionistic logic that is
preserved under bi-asimulations, then for every pointed
$\Theta$-model $(\mathcal{M}, w)$, for every $v \in W$ such that
$w\mathrel{\prec}v$, and for every $\phi \in L(\Theta)$ it is true
that:
$$
\mathcal{M}, w \models_\mathcal{L} \phi \Rightarrow \mathcal{M}, v
\models_\mathcal{L} \phi.
$$
\end{lemma}
\begin{proof}
We define a bi-asimulation $A$ from $(\mathcal{M}, w)$ to
$(\mathcal{M}, v)$ setting:
$$
A:= \{ (w, v) \} \cup id_W.
$$
We show that $A$ is indeed a bi-asimulation. Clearly, we only need to check conditions \eqref{E:c44} and \eqref{E:c55}. Moreover, if $s\mathrel{A}t$ are such that $s = t$, then even the latter two conditions are clearly satisfied. There remains the case when $s = w$ and $t = v$. In this case, if $v\mathrel{\prec}u$, then also $w\mathrel{\prec}u$ by transitivity and we can also choose $u$ as a successor for $w$. Symmetrically, if $u\mathrel{\prec}w$, then, by transitivity, also $u\mathrel{\prec}v$; thus we can also choose $u$ as a predecessor for $v$.

The lemma then follows by the preservation of $\mathcal{L}$ under bi-asimulations.
\end{proof}
Note that for Lemma \ref{L:monotonicity} one does not even need to
assume that $\mathsf{BIL} \sqsubseteq \mathcal{L}$.
The next lemma sums up some well-known facts about
$\mathsf{BIL}$:
\begin{lemma}\label{L:properties}
$\mathsf{BIL}$ is preserved under bi-asimulations, $\star$-compact and
has TUP.
\end{lemma}
\begin{proof}
Invariance under bi-asimulations follows from the main result of
\cite{ba1}. The $\star$-compactness of $\mathsf{BIL}$ is follows from the fact that it is representable as a fragment of classical first-order logic by the appropriate standard translation (also in \cite{ba1}). The proof of TUP runs along the lines of standard proofs of TUP; we briefly sketch it here.

We need to show that if $\varphi \in BIL(\Theta)$, $n \in \omega$
and $w \in W_n$, then
$$
\mathcal{M}_n, w \models_{BIL} \varphi \Leftrightarrow \bigcup_{n
\in \omega}\mathcal{M}_n , w \models_{BIL} \varphi;
$$
this is done by induction on the construction of $\varphi$, and the only non-trivial
cases are is when $\varphi = \psi \to \chi$ and when $\varphi = \psi \ll \chi$. 

\textit{Case 1}. Assume that  $\varphi = \psi \to \chi$. If $\mathcal{M}_n, w
\not\models_{BIL} \psi \to \chi$, then there is a $v \in W_n$ such
that $w\mathrel{\prec_n}v$ and $\mathcal{M}_n, v \models_{BIL}
(\{\psi\}, \{\chi\})$. But then, by induction hypothesis we must
have $\bigcup_{n \in \omega}\mathcal{M}_n, v \models_{BIL}
(\{\psi\}, \{\chi\})$ and we also have $w\mathrel{(\bigcup_{n \in
\omega}\prec_n)}v$ so that $\bigcup_{n \in \omega}\mathcal{M}_n, w
\not\models_{BIL} \psi \to \chi$. In the other direction, assume
that $\bigcup_{n \in \omega}\mathcal{M}_n, w \not\models_{BIL} \psi
\to \chi$. Then, for some $v \in \bigcup_{n \in \omega}W_n$ such
that $w\mathrel{(\bigcup_{n \in \omega}\prec_n})v$ it is true that
$\bigcup_{n \in \omega}\mathcal{M}_n, v \models_{BIL} (\{\psi\},
\{\chi\})$. But then, for some $k \geq n$, we must have both $w,v
\in W_k$ and $w\mathrel{\prec_k}v$, so that we get $\mathcal{M}_k, w
\not\models_{BIL} \psi \to \chi$. By obvious transitivity of
$\preccurlyeq_{BIL}$ we get then that $\mathcal{M}_n
\preccurlyeq_{BIL} \mathcal{M}_k$, whence $\mathcal{M}_n, w
\not\models_{BIL} \psi \to \chi$.

\textit{Case 2}. Assume that  $\varphi = \psi \ll \chi$. If $\mathcal{M}_n, w
\models_{BIL} \psi \ll \chi$, then there is a $v \in W_n$ such
that $v\mathrel{\prec_n}w$ and $\mathcal{M}_n, v \models_{BIL}
(\{\psi\}, \{\chi\})$. But then, by induction hypothesis we must
have $\bigcup_{n \in \omega}\mathcal{M}_n, v \models_{BIL}
(\{\psi\}, \{\chi\})$ and we also have $v\mathrel{(\bigcup_{n \in
		\omega}\prec_n)}w$ so that $\bigcup_{n \in \omega}\mathcal{M}_n, w
\models_{BIL} \psi \ll \chi$. In the other direction, assume
that $\bigcup_{n \in \omega}\mathcal{M}_n, w \models_{BIL} \psi
\ll \chi$. Then, for some $v \in \bigcup_{n \in \omega}W_n$ such
that $v\mathrel{(\bigcup_{n \in \omega}\prec_n})w$ it is true that
$\bigcup_{n \in \omega}\mathcal{M}_n, v \models_{BIL} (\{\psi\},
\{\chi\})$. But then, for some $k \geq n$, we must have both $w,v
\in W_k$ and $v\mathrel{\prec_k}w$, so that we get $\mathcal{M}_k, w
\models_{BIL} \psi \ll \chi$. By obvious transitivity of
$\preccurlyeq_{BIL}$ we get then that $\mathcal{M}_n
\preccurlyeq_{BIL} \mathcal{M}_k$, whence $\mathcal{M}_n, w
\models_{BIL} \psi \ll \chi$.
\end{proof}
The abstract bi-intuitionistic logics extending $\mathsf{BIL}$ and displaying the combination of all three useful model-theoretic properties mentioned in Lemma \ref{L:properties}, indeed, display a very regular behavior. In particular, one easily checks that an obvious analogue of every Lemma and Corollary established in Sections \ref{S:Prel}--\ref{S:unravel} above can be established for any such logic without changing anything essential in the proofs.

Our main theorem is then that no proper extension of $\mathsf{BIL}$
displays the combination of useful properties established in Lemma
\ref{L:properties}. In other words, we are going to establish the
following:
\begin{theorem}\label{L:main}
Let $\mathcal{L}$ be an abstract bi-intuitionistic logic. If
$\mathsf{BIL} \sqsubseteq \mathcal{L}$ and $\mathcal{L}$ is
preserved under bi-asimulations, $\star$-compact, and has the TUP,
then $\mathsf{BIL} \equiv \mathcal{L}$.
\end{theorem}

\section{The proof of Theorem \ref{L:main}}\label{main}
Before we start with the proof, we need one more piece of
notation. If $\mathcal{L}$ is an abstract bi-intuitionistic logic,
$\Theta$ a signature, and $\Gamma \subseteq L(\Theta)$, then we
let $Mod_\mathcal{L}(\Theta,\Gamma)$ denote the class of pointed
$\Theta$-models $(\mathcal{N}, u)$ such that for every $\phi \in
\Gamma$ it is true that:
$$
\mathcal{N}, u \models_\mathcal{L} \phi.
$$
If $\Gamma = \{ \phi \}$ for some $\phi \in L(\Theta)$ then
instead of $Mod_\mathcal{L}(\Theta,\Gamma)$ we simply write
$Mod_\mathcal{L}(\Theta,\phi)$.

We now start by establishing a couple of technical facts first:

\begin{proposition}\label{L:proposition1}
Let $\mathcal{L}$ be a $\star$-compact  abstract bi-intuitionistic
logic extending $\mathsf{BIL}$. Suppose that $\mathsf{BIL}
\not\equiv\mathcal{L}$. Then, there are $\phi \in L(\Theta_\phi)$
and pointed $\Theta_\phi$-models $(\mathcal{M}_1, w_1)$,
$(\mathcal{M}_2, w_2)$ such that $Th^+_{BIL}(\mathcal{M}_1, w_1)
\subseteq Th^+_{BIL}(\mathcal{M}_2, w_2)$ while $\mathcal{M}_1, w_1
\models_\mathcal{L} \phi$ and $\mathcal{M}_2, w_2
\not\models_\mathcal{L} \phi$.
\end{proposition}

\begin{proof}
Suppose that for an arbitrary $\phi \in L(\Theta_\phi)$ we have
shown that:

\begin{center}
\begin{itemize}
\item[(i)]  $Mod_\mathcal{L}(\Theta_\phi,\phi) =
\bigcup_{(\mathcal{N}, u) \in Mod_\mathcal{L}(\Theta_\phi,\phi)}
Mod_\mathcal{L}(\Theta_\phi,Th^+_{BIL}(\mathcal{N}, u))$,

\end{itemize}
\end{center}
Let  $(\mathcal{N}, u) \in Mod_\mathcal{L}(\Theta_\phi,\phi)$ be
arbitrary. The above implies that every $\Theta_\phi$-model of
$Th^+_{BIL}(\mathcal{N}, u)$ must be a model of $\phi$. But then
the theory $ (Th^+_{BIL}(\mathcal{N}, u), \{ \phi \})$ is
$\mathcal{L}$-unsatisfiable. By the $\star$-compactness of
$\mathcal{L}$, for some finite $\Psi_{(\mathcal{N}, u)} \subseteq
Th^+_{BIL}(\mathcal{N}, u)$ (and we can pick a unique one using the
Axiom of Choice if need be), the theory $(\Psi_{(\mathcal{N}, u)} ,\{
\phi\})$ is $\mathcal{L}$-unsatisfiable. Hence, $\bigwedge
\Psi_{(\mathcal{N}, u)} $ logically implies $\phi$ in
$\mathcal{L}$. Then, given (i), we get that
\begin{center}
\begin{itemize}
 \item [(ii)] $Mod_\mathcal{L}(\Theta_\phi,\phi) = \bigcup_{(\mathcal{N}, u) \in Mod_\mathcal{L}(\Theta_\phi,\phi)}Mod_\mathcal{L}(\Theta_\phi,\Psi_{(\mathcal{N}, u)})$.
\end{itemize}
\end{center}
However, this means that the theory $(\{ \phi \}, \{ \bigwedge
\Psi_{(\mathcal{N}, u)} \mid (\mathcal{N}, u) \in
Mod_\mathcal{L}(\Theta_\phi,\phi) \})$ is
$\mathcal{L}$-unsatisfiable and by (i), for some finite $\Gamma
\subseteq  \{ \bigwedge \Psi_{(\mathcal{N}, u)} \mid (\mathcal{N},
u) \in Mod_\mathcal{L}(\Theta_\phi,\phi) \}$, the theory $(\{ \phi
\}, \Gamma)$ is $\mathcal{L}$-unsatisfiable. This means that  $
Mod_\mathcal{L}(\Theta_\phi,\phi) \subseteq
Mod_\mathcal{L}(\Theta_\phi,\bigvee\Gamma)$. So, using (ii), since
clearly $Mod_\mathcal{L}(\Theta_\phi,\bigvee\Gamma) \subseteq
\bigcup_{(\mathcal{N}, u) \in
Mod_\mathcal{L}(\Theta_\phi,\phi)}Mod_\mathcal{L}(\Theta_\phi,\Psi_{(\mathcal{N},
u)})$, we get that:
\begin{center}
\begin{itemize}
 \item [(iii)] $Mod_\mathcal{L}(\Theta_\phi,\phi) = Mod_\mathcal{L}(\Theta_\phi,\bigvee\Gamma)$.
\end{itemize}
\end{center}
Now, $ \bigvee \Gamma$ is a perfectly good formula of
$BIL(\Theta_\phi)$ involving only finitary conjunctions and
disjunctions. So we have shown that every $\phi \in
L(\Theta_\phi)$ is just a bi-intuitionistic $\Theta_\phi$-formula
and hence that $\mathcal{L} \equiv \mathsf{BIL}$ which is in
contradiction with the hypothesis of the proposition.

Therefore, (i) must fail for at least one $\phi \in
L(\Theta_\phi)$, and clearly, for this $\phi$ it can only fail if
$$
\bigcup_{(\mathcal{N}, u) \in Mod_\mathcal{L}(\Theta_\phi,\phi)}
Mod_\mathcal{L}(\Theta_\phi,Th^+_{BIL}(\mathcal{N}, u))
\not\subseteq Mod_\mathcal{L}(\Theta_\phi,\phi).
$$
 But the latter
means that for some pointed intuitionistic $\Theta_\phi$-model
$(\mathcal{M}_1, w_1)$ such that $\mathcal{M}_1, w_1
\models_\mathcal{L}\phi$ there is another $\Theta_\phi$-model
$(\mathcal{M}_2, w_2)$ such that both $\mathcal{M}_2, w_2
\not\models_\mathcal{L}\phi$ and $Th^+_{BIL}(\mathcal{M}_1, w_1)$
is satisfied at $(\mathcal{M}_2, w_2)$. The latter means, in turn,
that we have $Th^+_{BIL}(\mathcal{M}_1, w_1) \subseteq
Th^+_{BIL}(\mathcal{M}_2, w_2)$ as desired.
\end{proof}

Assume $\mathcal{M}$ is a $\Theta$-model. We let 
$\Theta_\mathcal{M}$ be $\Theta \cup \{ q^+_w, q^-_w \mid w \in W
\}$ such that $\Theta \cap \{ q^+_w, q^-_w \mid w \in W \} =
\varnothing$, and we define that $[\mathcal{M}] = (W, \prec, [V])$ is
the $\Theta_\mathcal{M}$-model, such that $W$ and $\prec$ are just
borrowed from $\mathcal{M}$ and $[V]$ coincides with $V$ on
elements of $\Theta$, whereas for arbitrary $v, w \in W$ we set
that  $v \in [V](q^+_w)$ iff $w\mathrel{\prec}v$ and $v \notin
[V](q^-_w)$ iff $v\mathrel{\prec}w$. It is straightforward to check that $[\mathcal{M}]$ is a
$\Theta_\mathcal{M}$-model.

\begin{lemma}\label{L:lemma1}
	Let $\mathcal{L}$ be an abstract
	bi-intuitionistic logic preserved under bi-asimulations and extending $\mathsf{BIL}$, let $(\mathcal{M},
	w)$ be a pointed bi-unravelled $\Theta$-model, and let $(\mathcal{N}, v)$ be
	another pointed $\Theta_{\mathcal{M}}$-model. Assume that
	$Th_\mathcal{L}(\mathcal{N}, v) =
	Th_\mathcal{L}([\mathcal{M}], w)$. Then there exists an
	$\mathcal{L}$-elementary embedding $f$ of $[\mathcal{M}]$ into
	$\mathcal{N}$ such that $f(w) = v$.
\end{lemma}
\begin{proof}
We proceed by induction on $k \geq 1$ and define an increasing
	chain of the form
	$$
	f_1 \subseteq\ldots \subseteq f_k \subseteq\ldots
	$$
	in such a way that, for each $k < \omega$, $f_k:\{\bar{u}_l \in
	W^{un}_w\mid l \leq k\} \to  U$ is an injective
	function. In the process of inductively defining $f_k$ for all $k < \omega$,
	we show that our construction satisfies the following conditions
	for all $1 \leq l,m \leq k < \omega$, and all $\bar{u}_l,
	\bar{v}_m, \bar{w}_k \in W^{un}_w$:
	\begin{enumerate}	
		\item[$(i)_k$] $Th_\mathcal{L}(\mathcal{N}, f_k(\bar{w}_k)) =
		Th_\mathcal{L}([\mathcal{M}], \bar{w}_k)$.
		
		\item[$(ii)_k$] $f_k(\bar{u}_l)\mathrel{\lhd} f_k(\bar{v}_m)\Leftrightarrow \bar{u}_l\mathrel{\prec^{un}_w}\bar{v}_m$.
	\end{enumerate}
	
	\emph{Induction basis}. $k = 1$. Then $\bar{w}_k = w$, so we have
	to define $f_1$ for this unique node. We set $f_1(w) := v$
	thus getting an injective function. Under
	these settings, it is true that
	\begin{align*}
		Th_\mathcal{L}(\mathcal{N}, f_1(w)) = Th_\mathcal{L}(\mathcal{N}, v) = Th_\mathcal{L}([\mathcal{M}], w)&&\text{(by the assumption of the lemma)}
	\end{align*}
	so
	$(i)_1$ is satisfied. Next, if $1 \leq l,m \leq 1$ and
	$\bar{u}_l, \bar{w}_m \in W^{un}_w$, then $l = m = 1$ and we get
	$\bar{u}_l = w = \bar{w}_m = \bar{w}_l$ so that both $\bar{u}_l\mathrel{\prec^{un}_w}\bar{w}_m$ and $f_1(\bar{u}_l) =
	v \mathrel{\lhd} v= f_1(\bar{w}_m)$ by reflexivity.
	Therefore, $(ii)_1$ is satisfied as well.
	
	\emph{Induction step}. $k = l + 1$ for some $l \geq 1$. Then we
	already have the chain of injections $f_1 \subseteq\ldots
	\subseteq f_l$ defined in
	such a way that $(i)_m-(ii)_m$ are satisfied by $f_m$ for all $m \leq l$.
	
	If $\bar{w}_k \in W^{un}_w$ is chosen arbitrarily, then we must
	have $\bar{w}_k = \bar{w}_l^\frown w_k$. We have then
	\begin{equation}\label{E:eq1}
		Th_\mathcal{L}(\mathcal{N}, f_l(\bar{w}_l)) = Th_\mathcal{L}([\mathcal{M}], \bar{w}_l)
	\end{equation}
	by $(i)_l$. Two cases are possible:
	
	\textit{Case 1}. $w_l\mathrel{\prec}w_k$. Then, by definition of $[\mathcal{M}]$ and by the closure of
	$\mathcal{L}$ w.r.t. intuitionistic implication, we have that the
	set:
	$$
	\Gamma_{\bar{w}_k} = \{ q^+_{\bar{w}_k} \to \phi \mid
	\phi \in Th^+_\mathcal{L}([\mathcal{M}], \bar{w}_k) \} \cup
	\{ \psi \to q^-_{\bar{w}_k} \mid \psi \in
	Th^-_\mathcal{L}([\mathcal{M}], \bar{w}_k) \}
	$$
	is a subset of $Th^+_\mathcal{L}([\mathcal{M}], \bar{w}_l)$, thus
	also of $Th^+_\mathcal{L}(\mathcal{N},
	f_l(\bar{w}_l))$. We also know that, by
	$\bar{w}_l\mathrel{\prec^{un}_w}\bar{w}_k$, and by the fact that
	the theory $(\{ q^+_{\bar{w}_k} \},\{ q^-_{\bar{w}_k}\})$ is $\mathcal{L}$-satisfied at
	$([\mathcal{M}], \bar{w}_k)$, we have:
	$$
	[\mathcal{M}], \bar{w}_l \not\models_\mathcal{L}
	q^+_{\bar{w}_k}\to q^-_{\bar{w}_k}.
	$$
	Therefore, by \eqref{E:eq1}, also $\mathcal{N}, f_l(\bar{w}_l)
	\not\models_\mathcal{L} q^+_{\bar{w}_k}\to q^-_{\bar{w}_k}$. So fix any $v'_{\bar{w}_k} \in U$ such that $f_l(\bar{w}_l)\mathrel{\lhd}v'_{\bar{w}_k}$ and $(\{ q^+_{\bar{w}_k} \},\{ q^-_{\bar{w}_k}\})$ is $\mathcal{L}$-satisfied at
	$(\mathcal{N}, v'_{\bar{w}_k})$. By Lemma \ref{L:monotonicity} and the preservation of $\mathcal{L}$ under
	bi-asimulations, we get that $\Gamma_{\bar{w}_k} \subseteq Th^+_\mathcal{L}(\mathcal{N},
	f_l(\bar{w}_l)) \subseteq Th^+_\mathcal{L}(\mathcal{N},
	v'_{\bar{w}_k})$ and thus, clearly, $Th_\mathcal{L}([\mathcal{M}], \bar{w}_k) = Th_\mathcal{L}(\mathcal{N}, v'_{\bar{w}_k})$.
	
	\textit{Case 2}. $w_l\mathrel{\succ}w_k$. Then, by definition of $[\mathcal{M}]$ and by the closure of
	$\mathcal{L}$ w.r.t. co-implication, we have that the
	set:
	$$
	\Delta_{\bar{w}_k} = \{ q^+_{\bar{w}_k} \ll \phi \mid
	\phi \in Th^+_\mathcal{L}([\mathcal{M}], \bar{w}_k) \} \cup
	\{ \psi \ll q^-_{\bar{w}_k} \mid \psi \in
	Th^-_\mathcal{L}([\mathcal{M}], \bar{w}_k) \}
	$$
	is a subset of $Th^-_\mathcal{L}([\mathcal{M}], \bar{w}_l)$, thus
	also of $Th^-_\mathcal{L}(\mathcal{N},
	f_l(\bar{w}_l))$. We also know that, by
	$\bar{w}_k\mathrel{\prec^{un}_w}\bar{w}_l$, and by the fact that
	the theory $(\{ q^+_{\bar{w}_k} \},\{ q^-_{\bar{w}_k}\})$ is $\mathcal{L}$-satisfied at
	$([\mathcal{M}], \bar{w}_k)$, we have:
	$$
	[\mathcal{M}], \bar{w}_l \models_\mathcal{L}
	q^+_{\bar{w}_k}\ll q^-_{\bar{w}_k}.
	$$
	Therefore, by \eqref{E:eq1}, also $\mathcal{N}, f_l(\bar{w}_l)
	\models_\mathcal{L} q^+_{\bar{w}_k}\ll q^-_{\bar{w}_k}$. So fix any $v'_{\bar{w}_k} \in U$ such that $f_l(\bar{w}_l)\mathrel{\rhd}v'_{\bar{w}_k}$ and $(\{ q^+_{\bar{w}_k} \},\{ q^-_{\bar{w}_k}\})$ is $\mathcal{L}$-satisfied at
	$(\mathcal{N}, v'_{\bar{w}_k})$. By Lemma \ref{L:monotonicity} and the preservation of $\mathcal{L}$ under
	bi-asimulations, we get that $\Delta_{\bar{w}_k} \subseteq Th^-_\mathcal{L}(\mathcal{N},
	f_l(\bar{w}_l)) \subseteq Th^-_\mathcal{L}(\mathcal{N},
	v'_{\bar{w}_k})$ and thus, clearly, $Th_\mathcal{L}([\mathcal{M}], \bar{w}_k) = Th_\mathcal{L}(\mathcal{N}, v'_{\bar{w}_k})$.
	
	Having thus decided on our choices for $v'_{\bar{w}_k}$ for every $\bar{w}_k\in W^{un}_w$, we now set:
	$$
	f_k(\bar{u}_m) :=\left\{%
	\begin{array}{ll}
		f_l(\bar{u}_m), & \hbox{if $m \leq l$;} \\
		v'_{\bar{u}_m}, & \hbox{if $m = k$.} \\
	\end{array}%
	\right.
	$$
	It is clear then that $f_k$ is a function from $\{\bar{u}_m \in
	W^{un}_w\mid m \leq k\}$ to $U$, since $f_l$ is a
	function and we have fixed a unique $v'_{\bar{w}_k} \in U$ for every $\bar{w}_k \in W^{un}_w$. It is also immediate
	to see that $(i)_k$ is satisfied.
	
	We now address the injectivity of $f_k$. Let $m \leq r \leq k$ and
	let $\bar{u}_m, \bar{w}_r \in W^{un}_w$ be such that $\bar{u}_m
	\neq \bar{w}_r$. Then we have either
	$\bar{u}_m\mathrel{\not\prec^{un}_w}\bar{w}_r$ or
	$\bar{w}_r\mathrel{\not\prec^{un}_w}\bar{u}_m$. In the first case,
	we have $q^+_{\bar{u}_m} \in
	Th^+_\mathcal{L}([\mathcal{M}], \bar{u}_m)\setminus
	Th^+_\mathcal{L}([\mathcal{M}], \bar{w}_r)$, so by
	$(i)_k$ we must have
	$$
	q^+_{\bar{u}_m} \in
	Th^+_\mathcal{L}(\mathcal{N},f_k(\bar{u}_m))
	\setminus Th^+_\mathcal{L}(\mathcal{N},
	f_k(\bar{w}_r)),
	$$
	whence $f_k(\bar{u}_m) \neq f_k(\bar{w}_k)$. In the second case, we
	reason symmetrically, using $q^+_{\bar{w}_r}$ instead
	of $q^+_{\bar{u}_m}$. Thus $f_k$ is injective.
	
	Next, we verify $(ii)_k$. Let $\bar{u}_m, \bar{w}_r \in W^{un}_w$ be such
	that $m, r \leq k$.
	
	($\Leftarrow$). If $\bar{u}_m\mathrel{\prec^{un}_w}\bar{w}_r$, then choose, by Lemma \ref{L:prec-relation}, $m_1, n_1 \geq 0$ and $\delta \in W^{un}_w$ and (not necessarily pairwise distinct) $\bar{s}_{m_1} \in W^{m_1}$ and $\bar{t}_{n_1} \in W^{n_1}$ such that $s_{m_1}\mathrel{\prec}\ldots\mathrel{\prec}s_1\mathrel{\prec}end(\delta)\mathrel{\prec}t_1\mathrel{\prec}\ldots\mathrel{\prec}t_{n_1}$, and:
	\begin{align*}
		\bar{u}_m &= \delta^\frown\bar{s}_{m_1}\\
		\bar{w}_r &=  \delta^\frown\bar{t}_{n_1}.
	\end{align*}
	The following cases are then possible:
	
	\textit{Case 1}. $m_1 = n_1 = 0$. Then $\bar{u}_m = \bar{w}_r$, thus also 
	$f(\bar{u}_m) = f(\bar{w}_r)$, and so $f(\bar{u}_m) \lhd f(\bar{w}_r)$ by reflexivity.
	
	\textit{Case 2}. $m_1, n_1 > 0$. But then also $\delta^\frown\bar{s}_{m_1- 1}, \delta^\frown\bar{t}_{n_1 -1} \in W^{un}_w$, and due to their smaller length, we also get that $\delta^\frown\bar{s}_{m_1- 1}, \delta^\frown\bar{t}_{n_1 -1} \in dom(f_k)$. Moreover, it follows from Lemma \ref{L:prec-relation} that we have:
	$$
	\bar{u}_m = \delta^\frown\bar{s}_{m_1}\mathrel{\prec^{un}_w}\delta^\frown\bar{s}_{m_1- 1}\mathrel{\prec^{un}_w}\delta^\frown\bar{t}_{n_1 -1}\mathrel{\prec^{un}_w}\delta^\frown\bar{t}_{n_1} = \bar{w}_r.
	$$
	We now claim the following:
	
	\textit{Claim}. We have:
	$$
	f_k(\bar{u}_m) = f_k(\delta^\frown\bar{s}_{m_1})\mathrel{\lhd}f_k(\delta^\frown\bar{s}_{m_1- 1})\mathrel{\lhd}f_k(\delta^\frown\bar{t}_{n_1 -1})\mathrel{\lhd}f_k(\delta^\frown\bar{t}_{n_1}) = f_k(\bar{w}_r).
	$$
	
	Indeed, if $m = k$, then we have $f_k(\bar{u}_m) = f_k(\delta^\frown\bar{s}_{m_1})\mathrel{\lhd}f_k(\delta^\frown\bar{s}_{m_1- 1})$ by the choice of $f_k(\bar{u}_m)$, otherwise, we get it by $(ii)_l$. Similarly, if $r = k$, then we get $f_k(\delta^\frown\bar{t}_{n_1 -1})\mathrel{\lhd}f_k(\delta^\frown\bar{t}_{n_1}) = f_k(\bar{w}_r)$ by the choice of $f_k(\bar{w}_r)$, otherwise we get the same result by $(ii)_l$. Finally, the length of both $\delta^\frown\bar{s}_{m_1- 1}$ and $\delta^\frown\bar{t}_{n_1 -1}$ cannot exceed $l$, therefore $f_k(\delta^\frown\bar{s}_{m_1- 1})\mathrel{\lhd}f_k(\delta^\frown\bar{t}_{n_1 -1})$ also follows from $(ii)_l$.
	
	The Claim is thus proven and the ($\Leftarrow$)-part for Case 2 clearly follows from the Claim.
	
	\textit{Case 3}. $m_1 = 0$, $n_1 > 0$. Then also $\bar{u}_m = \delta$, $\delta^\frown\bar{t}_{n_1 -1} \in W^{un}_w$, and both of these sequences have a length that does not exceed $l$. We have then,  by Lemma \ref{L:prec-relation} that: $
	\bar{u}_m = \delta\mathrel{\prec^{un}_w}\delta^\frown\bar{t}_{n_1 -1}\mathrel{\prec^{un}_w}\delta^\frown\bar{t}_{n_1} = \bar{w}_r$. We immediately get that $f_k(\bar{u}_m) = f_l(\bar{u}_m)\lhd f_l(\delta^\frown\bar{t}_{n_1}) = f_k(\delta^\frown\bar{t}_{n_1})$ by $(ii)_l$. Now, if $r = k$, then we get $f_k(\delta^\frown\bar{t}_{n_1 -1})\mathrel{\lhd}f_k(\delta^\frown\bar{t}_{n_1}) = f_k(\bar{w}_r)$ by the choice of $f_k(\bar{w}_r)$, otherwise we get the same result by $(ii)_l$. Summing up, we get that $f_k(\bar{u}_m)\mathrel{\lhd} f_k(\delta^\frown\bar{t}_{n_1})\mathrel{\lhd} f_k(\bar{w}_r)$, hence also $f_k(\bar{u}_m)\mathrel{\lhd} f_k(\bar{w}_r)$ by transitivity.
	
	\textit{Case 4}. $m_1 > 0$, $n_1 = 0$. This case is symmetric to Case 4.
	
	($\Rightarrow$). If $f_k(\bar{u}_m)\mathrel{\lhd} f_k(\bar{w}_r)$, then we have $q^+_{\bar{u}_m} \in Th^+_\mathcal{L}([\mathcal{M}],
	\bar{u}_m) = Th^+_\mathcal{L}(\mathcal{N},
	f_k(\bar{u}_m))$ by $(i)_k$, but then, by Lemma
	\ref{L:monotonicity}, preservation of $\mathcal{L}$ under
	bi-asimulations, and $(i)_k$ again, we further get that
	$$
	q^+_{\bar{u}_m} \in
	Th^+_\mathcal{L}(\mathcal{N}, f_k(\bar{w}_r)) =
	Th^+_\mathcal{L}([\mathcal{M}], \bar{w}_r),
	$$
	whence we
	must have $\bar{u}_m\mathrel{\prec^{un}_w}\bar{w}_r$ by the definition of $[\mathcal{M}]$.
	
	In this way, $(ii)_k$ is shown to hold and the Induction Step is proven.
	
	We now define:
	\begin{align*}
		f: = \bigcup\{ f_k\mid k < \omega \}
	\end{align*}
	We will show that $f$ is an $\mathcal{L}$-elementary embedding of
	$[\mathcal{M}]$ into $\mathcal{N}$. It is
	obvious that $f: W^{un}_w \to U$ is an injection,
	being a union of the countable increasing chain of injections such
	that the sequence of the domains of these injections covers all of
	$W^{un}_w$.
	
	It remains to check that the conditions of Definition
	\ref{D:embedding} are satisfied by $f$. As for
	condition \eqref{E:c1}, assume that $\bar{u}_m,\bar{w}_k \in
	W^{un}_w$. Then we have
	\begin{align*}
		\bar{u}_m\mathrel{\prec^{un}_w}\bar{w}_k &\Leftrightarrow f_{max(m,k)}(\bar{u}_m)\mathrel{\lhd}f_{max(m,k)}(\bar{w}_k)&&\text{(by $(ii)_{max(m,k)}$)}\\
		&\Leftrightarrow f(\bar{u}_m)\mathrel{\lhd}f(\bar{w}_k)&&\text{(by definition of $f$)}
	\end{align*}
	Finally, as for the condition \eqref{E:c2}, we
	observe that for
	any given $\bar{u}_m\in W^{un}_w$ and any $\phi \in
	L(\Theta_{\mathcal{M}})$ we have:
	\begin{align*}
		Th_{\mathcal{L}}([\mathcal{M}], \bar{u}_m) &= Th_{\mathcal{L}}(\mathcal{N}, f_m(\bar{u}_m)) &&\text{(by $(i)_m$)}\\
		&= Th_{\mathcal{L}}(\mathcal{N}, f(\bar{u}_m)) &&\text{(by definition of $f$)}
	\end{align*}
\end{proof}
More generally, extending the signature of a bi-unravelled $(\mathcal{M}, w) \in Pmod_\Theta$ to get $[\mathcal{M}]$ allows us to force many types of formulas holding and failing at every node of $[\mathcal{M}]$ while staying at $w$. This functionality even extends to arbitrary models for which a suitable elementary embedding of $[\mathcal{M}]$ into them is possible. More precisely, the following lemma holds:

\begin{lemma}\label{L:representation-emb}
	Let $\mathcal{L}$ be an abstract bi-intuitionistic logic which extends $\mathsf{BIL}$, is
	preserved under bi-asimulations, $\star$-compact, and has TUP, let $(\mathcal{M}, w) \in Pmod_\Theta$ be a bi-unravelled model, and let $f$ be an $\mathcal{L}$-elementary embedding  of $[\mathcal{M}]$ into
	$\mathcal{N}$. Then, for every $n > 0$, every $\bar{w}_n \in W^{un}_w$, and every $1 \leq k \leq n$ there exist binary formula schemas $\phi^k_{\bar{w}_n}(\underline{\hspace{0.3cm}}_1,\underline{\hspace{0.3cm}}_2), \psi^k_{\bar{w}_n}(\underline{\hspace{0.3cm}}_1,\underline{\hspace{0.3cm}}_2), \theta^k_{\bar{w}_n}(\underline{\hspace{0.3cm}}_1,\underline{\hspace{0.3cm}}_2), \tau^k_{\bar{w}_n}(\underline{\hspace{0.3cm}}_1,\underline{\hspace{0.3cm}}_2) \in BIL(\Theta_{\mathcal{M}})$, and a function 
	$$
	F:\{\phi^k_{\bar{w}_n}, \psi^k_{\bar{w}_n}, \theta^k_{\bar{w}_n}, \tau^k_{\bar{w}_n}\mid \bar{w}_n \in W^{un}_w,\, 1 \leq k \leq n\} \to \{+, -\}
	$$
	such that, if $1 \leq k \leq n$ , then, for every $\Sigma \supseteq \Theta_{\mathcal{M}}$, every $\Sigma$-expansion $\mathcal{N}'$ of $\mathcal{N}$, and all $\alpha, \beta \in L(\Sigma)$,  we have:
	\begin{enumerate}
		\item If $F(\phi^k_{\bar{w}_n}) = +$ (resp. $F(\phi^k_{\bar{w}_n}) = -$) and $\mathcal{N}', f(\bar{w}_k)\models_{\mathcal{L}} \phi^k_{\bar{w}_n}(\alpha, \beta)$ (resp. $\mathcal{N}', f(\bar{w}_k)\not\models_{\mathcal{L}} \phi^k_{\bar{w}_n}(\alpha, \beta)$), then $\mathcal{N}', f(\bar{w}_n)\not\models_{\mathcal{L}} \alpha\to \beta$.
		
		\item If $F(\psi^k_{\bar{w}_n}) = +$ (resp. $F(\psi^k_{\bar{w}_n}) = -$) and $\mathcal{N}', f(\bar{w}_k)\models_{\mathcal{L}} \psi^k_{\bar{w}_n}(\alpha, \beta)$ (resp. $\mathcal{N}', f(\bar{w}_k)\not\models_{\mathcal{L}} \psi^k_{\bar{w}_n}(\alpha, \beta)$), then $\mathcal{N}', f(\bar{w}_n)\models_{\mathcal{L}} \alpha \ll \beta$.
		
		\item If $F(\theta^k_{\bar{w}_n}) = +$ (resp. $F(\theta^k_{\bar{w}_n}) = -$) and $\mathcal{N}', f(\bar{w}_k)\models_{\mathcal{L}} \theta^k_{\bar{w}_n}(\alpha, \beta)$ (resp. $\mathcal{N}', f(\bar{w}_k)\not\models_{\mathcal{L}} \theta^k_{\bar{w}_n}(\alpha, \beta)$), then $\mathcal{N}', f(\bar{w}_n)\models_{\mathcal{L}} \alpha\to \beta$.
		
		\item If $F(\tau^k_{\bar{w}_n}) = +$ (resp. $F(\tau^k_{\bar{w}_n}) = -$) and $\mathcal{N}', f(\bar{w}_k)\models_{\mathcal{L}} \tau^k_{\bar{w}_n}(\alpha, \beta)$ (resp. $\mathcal{N}', f(\bar{w}_k)\not\models_{\mathcal{L}} \tau^k_{\bar{w}_n}(\alpha, \beta)$), then $\mathcal{N}', f(\bar{w}_n)\not\models_{\mathcal{L}} \alpha\ll \beta$.
	\end{enumerate}   
\end{lemma}
\begin{proof}
	We fix a $\bar{w}_n \in W^{un}_w$ and argue by induction on $t := n - k$ so that we have $0 \leq t \leq n - 1$.
	
	\textit{Induction Basis}. $t = 0$. Then $k = n$ and we set:
	\begin{align*}
		\phi^n_{\bar{w}_n}(\alpha, \beta) &:= \alpha\to \beta &&F(\phi^n_{\bar{w}_n}) = -\\
		\psi^n_{\bar{w}_n}(\alpha, \beta) &:= \alpha\ll \beta &&F(\psi^n_{\bar{w}_n}) = +\\
		\theta^n_{\bar{w}_n}(\alpha, \beta) &:= \alpha\to \beta &&F(\theta^n_{\bar{w}_n}) = +\\
		\tau^n_{\bar{w}_n}(\alpha, \beta) &:= \alpha\ll \beta &&F(\tau^n_{\bar{w}_n}) = -	
	\end{align*}
	It is evident that the statement of the Lemma holds under these settings.
	
	\textit{Induction Step}. Let $t = l + 1$, so that $k = n - l - 1$. Applying the Induction Hypothesis for $t = l$ with $k = n - l$, we find the formula schemas $\phi^{n - l}_{\bar{w}_n}$, $\psi^{n - l}_{\bar{w}_n}$, $\theta^{n - l}_{\bar{w}_n}$, and $\tau^{n - l}_{\bar{w}_n}$, and $F(\phi^{n - l}_{\bar{w}_n})$, $F(\psi^{n - l}_{\bar{w}_n})$, $F(\theta^{n - l}_{\bar{w}_n})$, and $F(\tau^{n - l}_{\bar{w}_n})$ satisfying the statement of the Lemma. We then have to distinguish between two cases.
	
	\textit{Case 1}. We have $w_{n - l -1}\mathrel{\prec}w_{n - l}$, thus also $\bar{w}_{n - l -1}\mathrel{\prec^{un}_w}\bar{w}_{n - l}$. We then set $F(\phi^{n - l-1}_{\bar{w}_n}) = F(\psi^{n - l-1}_{\bar{w}_n}) = F(\theta^{n - l-1}_{\bar{w}_n}) = F(\tau^{n - l-1}_{\bar{w}_n}) := +$, and, as for the formulas, we set that
	\begin{align*}
		\chi^{n - l-1}_{\bar{w}_n}(\alpha, \beta) := \begin{cases}
			q^+_{\bar{w}_{n - l}} \to \chi^{n - l}_{\bar{w}_n}(\alpha, \beta),\text{ if $F(\chi^{n - l}_{\bar{w}_n}) = +$}\\
		\chi^{n - l}_{\bar{w}_n}(\alpha, \beta)\to q^-_{\bar{w}_{n - l}},\text{ if $F(\chi^{n - l}_{\bar{w}_n}) = -$}	
		\end{cases}
	\end{align*}
	for every $\chi \in \{\phi, \psi, \theta, \tau\}$. We show that our settings satisfy the Lemma. 
	
	\textit{Case 1.1}. Let $\chi \in \{\phi, \psi, \theta, \tau\}$ and assume that $F(\chi^{n - l}_{\bar{w}_n}) = +$. Then, if $\mathcal{N}', f(\bar{w}_{n - l-1})\models_{\mathcal{L}} 	q^+_{\bar{w}_{n - l}} \to \chi^{n - l}_{\bar{w}_n}(\alpha, \beta)$, note that we must have both  $f(\bar{w}_{n - l -1})\mathrel{\lhd}f(\bar{w}_{n - l})$, by $\bar{w}_{n - l -1}\mathrel{\prec^{un}_w}\bar{w}_{n - l}$ and condition \eqref{E:c1} of Definition \ref{D:embedding}, and also $\mathcal{N}', f(\bar{w}_{n - l})\models_{\mathcal{L}} q^+_{\bar{w}_{n - l}}$, by $[\mathcal{M}], \bar{w}_{n - l}\models_{\mathcal{L}}  q^+_{\bar{w}_{n - l}}$ and condition \eqref{E:c2} of Definition \ref{D:embedding}. Therefore, we must also have $\mathcal{N}', f(\bar{w}_{n - l})\models_{\mathcal{L}} 	\chi^{n - l}_{\bar{w}_n}(\alpha, \beta)$, so that we can apply Induction Hypothesis for $l$.
	
	\textit{Case 1.2}. Let $\chi \in \{\phi, \psi, \theta, \tau\}$ and assume that $F(\chi^{n - l}_{\bar{w}_n}) = -$. Then, if $\mathcal{N}', f(\bar{w}_{n - l-1})\models_{\mathcal{L}} 	 \chi^{n - l}_{\bar{w}_n}(\alpha, \beta)\to q^-_{\bar{w}_{n - l}}$, note that we must have both  $f(\bar{w}_{n - l -1})\mathrel{\lhd}f(\bar{w}_{n - l})$, by $\bar{w}_{n - l -1}\mathrel{\prec^{un}_w}\bar{w}_{n - l}$ and condition \eqref{E:c1} of Definition \ref{D:embedding}, and also $\mathcal{N}', f(\bar{w}_{n - l})\not\models_{\mathcal{L}} q^-_{\bar{w}_{n - l}}$, by $[\mathcal{M}], \bar{w}_{n - l}\not\models_{\mathcal{L}}  q^-_{\bar{w}_{n - l}}$ and condition \eqref{E:c2} of Definition \ref{D:embedding}. Therefore, we must also have $\mathcal{N}', f(\bar{w}_{n - l})\not\models_{\mathcal{L}} 	\chi^{n - l}_{\bar{w}_n}(\alpha, \beta)$, so that we can apply Induction Hypothesis for $l$.
	
	\textit{Case 2}. We have $w_{n - l -1}\mathrel{\succ}w_{n - l}$, thus also $\bar{w}_{n - l}\mathrel{\prec^{un}_w}\bar{w}_{n - l -1}$. We then set $F(\phi^{n - l-1}_{\bar{w}_n}) = F(\psi^{n - l-1}_{\bar{w}_n}) = F(\theta^{n - l-1}_{\bar{w}_n}) = F(\tau^{n - l-1}_{\bar{w}_n}) := -$, and, as for the formulas, we set that
	\begin{align*}
		\chi^{n - l-1}_{\bar{w}_n}(\alpha, \beta) := \begin{cases}
			q^+_{\bar{w}_{n - l}} \ll \chi^{n - l}_{\bar{w}_n}(\alpha, \beta),\text{ if $F(\chi^{n - l}_{\bar{w}_n}) = +$}\\
			\chi^{n - l}_{\bar{w}_n}(\alpha, \beta)\ll q^-_{\bar{w}_{n - l}},\text{ if $F(\chi^{n - l}_{\bar{w}_n}) = -$}	
		\end{cases}
	\end{align*}
for every $\chi \in \{\phi, \psi, \theta, \tau\}$. We show that our settings satisfy the Lemma. 

		\textit{Case 2.1}. Let $\chi \in \{\phi, \psi, \theta, \tau\}$ and assume that $F(\chi^{n - l}_{\bar{w}_n}) = +$. Then, if $\mathcal{N}', f(\bar{w}_{n - l-1})\not\models_{\mathcal{L}} 	q^+_{\bar{w}_{n - l}} \ll \chi^{n - l}_{\bar{w}_n}(\alpha, \beta)$, note that we must have both  $f(\bar{w}_{n - l})\mathrel{\lhd}f(\bar{w}_{n - l -1})$, by $\bar{w}_{n - l}\mathrel{\prec^{un}_w}\bar{w}_{n - l -1}$ and condition \eqref{E:c1} of Definition \ref{D:embedding}, and also $\mathcal{N}', f(\bar{w}_{n - l})\models_{\mathcal{L}} q^+_{\bar{w}_{n - l}}$, by $[\mathcal{M}], \bar{w}_{n - l}\models_{\mathcal{L}}  q^+_{\bar{w}_{n - l}}$ and condition \eqref{E:c2} of Definition \ref{D:embedding}. Therefore, we must also have $\mathcal{N}', f(\bar{w}_{n - l})\models_{\mathcal{L}} 	\chi^{n - l}_{\bar{w}_n}(\alpha, \beta)$, so that we can apply Induction Hypothesis for $l$.
	
	\textit{Case 2.2}. Let $\chi \in \{\phi, \psi, \theta, \tau\}$ and assume that $F(\chi^{n - l}_{\bar{w}_n}) = -$. Then, if $\mathcal{N}', f(\bar{w}_{n - l-1})\not\models_{\mathcal{L}} 	 \chi^{n - l}_{\bar{w}_n}(\alpha, \beta)\ll q^-_{\bar{w}_{n - l}}$, note that we must have both $f(\bar{w}_{n - l})\mathrel{\lhd}f(\bar{w}_{n - l -1})$, by $\bar{w}_{n - l}\mathrel{\prec^{un}_w}\bar{w}_{n - l -1}$ and condition \eqref{E:c1} of Definition \ref{D:embedding}, and also $\mathcal{N}', f(\bar{w}_{n - l})\not\models_{\mathcal{L}} q^-_{\bar{w}_{n - l}}$, by $[\mathcal{M}], \bar{w}_{n - l}\not\models_{\mathcal{L}}  q^-_{\bar{w}_{n - l}}$ and condition \eqref{E:c2} of Definition \ref{D:embedding}. Therefore, we must also have $\mathcal{N}', f(\bar{w}_{n - l})\not\models_{\mathcal{L}} 	\chi^{n - l}_{\bar{w}_n}(\alpha, \beta)$ so that we can apply Induction Hypothesis for $l$.
\end{proof}
Furthermore, if in the above Lemma we replace $\mathcal{N}$ with $[\mathcal{M}]$, then the conditionals in the lemma can be strengthened to bi-conditionals:
 \begin{lemma}\label{L:representation-m}
 	Let $\mathcal{L}$ be an abstract bi-intuitionistic logic which extends $\mathsf{BIL}$, is
 	preserved under bi-asimulations, $\star$-compact, and has TUP, and let $(\mathcal{M}, w) \in Pmod_\Theta$ be a bi-unravelled model. Then, for every $n > 0$, every $\bar{w}_n \in W^{un}_w$, and every $1 \leq k \leq n$ let the binary formula schemas $\phi^k_{\bar{w}_n}(\underline{\hspace{0.3cm}}_1,\underline{\hspace{0.3cm}}_2), \psi^k_{\bar{w}_n}(\underline{\hspace{0.3cm}}_1,\underline{\hspace{0.3cm}}_2), \theta^k_{\bar{w}_n}(\underline{\hspace{0.3cm}}_1,\underline{\hspace{0.3cm}}_2), \tau^k_{\bar{w}_n}(\underline{\hspace{0.3cm}}_1,\underline{\hspace{0.3cm}}_2) \in BIL(\Theta_{\mathcal{M}})$, and the function 
 	$$
 	F:\{\phi^k_{\bar{w}_n}, \psi^k_{\bar{w}_n}, \theta^k_{\bar{w}_n}, \tau^k_{\bar{w}_n}\mid \bar{w}_n \in W^{un}_w,\, 1 \leq k \leq n\} \to \{1, -1\}
 	$$
 	be defined as in Lemma \ref{L:representation-emb}. Then, for every $\Sigma \supseteq \Theta_{\mathcal{M}}$, every $\Sigma$-expansion $\mathcal{M}'$ of $[\mathcal{M}]$, and all $\alpha, \beta \in L(\Sigma)$,  we have:
 	\begin{enumerate}
 		\item If $F(\phi^k_{\bar{w}_n}) = +$ (resp. $F(\phi^k_{\bar{w}_n}) = -$) then $\mathcal{M}', \bar{w}_k\models_{\mathcal{L}} \phi^k_{\bar{w}_n}(\alpha, \beta)$ (resp. $\mathcal{M}', \bar{w}_k\not\models_{\mathcal{L}} \phi^k_{\bar{w}_n}(\alpha, \beta)$) iff $\mathcal{M}', \bar{w}_n\not\models_{\mathcal{L}} \alpha\to \beta$.
 		
 		\item If $F(\psi^k_{\bar{w}_n}) = +$ (resp. $F(\psi^k_{\bar{w}_n}) = -$), then $\mathcal{M}', \bar{w}_k\models_{\mathcal{L}} \psi^k_{\bar{w}_n}(\alpha, \beta)$ (resp. $\mathcal{M}', \bar{w}_k\not\models_{\mathcal{L}} \psi^k_{\bar{w}_n}(\alpha, \beta)$) iff $\mathcal{M}', \bar{w}_n\models_{\mathcal{L}} \alpha\ll \beta$.
 		
 		\item If $F(\theta^k_{\bar{w}_n}) = +$ (resp. $F(\theta^k_{\bar{w}_n}) = -$), then $\mathcal{M}', \bar{w}_k\models_{\mathcal{L}} \theta^k_{\bar{w}_n}(\alpha, \beta)$ (resp. $\mathcal{N}', \bar{w}_k\not\models_{\mathcal{L}} \theta^k_{\bar{w}_n}(\alpha, \beta)$) iff $\mathcal{M}', \bar{w}_n\models_{\mathcal{L}} \alpha\to \beta$.
 		
 		\item If $F(\tau^k_{\bar{w}_n}) = +$ (resp. $F(\tau^k_{\bar{w}_n}) = -$), then $\mathcal{M}', \bar{w}_k\models_{\mathcal{L}} \tau^k_{\bar{w}_n}(\alpha, \beta)$ (resp. $\mathcal{M}', \bar{w}_k\not\models_{\mathcal{L}} \tau^k_{\bar{w}_n}(\alpha, \beta)$) iff $\mathcal{M}', \bar{w}_n\not\models_{\mathcal{L}} \alpha\ll\beta$.
 	\end{enumerate}   
 \end{lemma}
 \begin{proof}
 	Since $id_{W^{un}_w}$ is, by the Isomorphism Property, an $\mathcal{L}$-elementary embedding of $[\mathcal{M}]$ into itself, we have the ($\Rightarrow$)-part of the Lemma by Lemma \ref{L:representation-emb}. As for the ($\Leftarrow$)-part, we, again, fix a $\bar{w}_n \in W^{un}_w$ and argue by induction on $t := n - k$ so that we have $0 \leq t \leq n - 1$. Induction basis is obvious, given the definitions in the proof of Lemma \ref{L:representation-emb}.
 	
 	\textit{Induction Step}. Let $t = l + 1$, so that $k = n - l - 1$. Then Induction Hypothesis for $t = l$ with $k = n - l$, tells us that the formulas $\phi^{n - l}_{\bar{w}_n}$, $\psi^{n - l}_{\bar{w}_n}$, $\theta^{n - l}_{\bar{w}_n}$, and $\tau^{n - l}_{\bar{w}_n}$, as well as $F(\phi^{n - l}_{\bar{w}_n})$, $F(\psi^{n - l}_{\bar{w}_n})$, $F(\theta^{n - l}_{\bar{w}_n})$, and $F(\tau^{n - l}_{\bar{w}_n})$ satisfy the statement of the Lemma. We then have to distinguish between two cases.
 	
 	\textit{Case 1}. We have $w_{n - l -1}\mathrel{\prec}w_{n - l}$, thus also $\bar{w}_{n - l -1}\mathrel{\prec^{un}_w}\bar{w}_{n - l}$. We then have that $F(\phi^{n - l-1}_{\bar{w}_n}) = F(\psi^{n - l-1}_{\bar{w}_n}) = F(\theta^{n - l-1}_{\bar{w}_n}) = F(\tau^{n - l-1}_{\bar{w}_n}) = +$, and, as for the formulas, that
 	\begin{align*}
 		\chi^{n - l-1}_{\bar{w}_n}(\alpha, \beta) = \begin{cases}
 			q^+_{\bar{w}_{n - l}} \to \chi^{n - l}_{\bar{w}_n}(\alpha, \beta),\text{ if $F(\chi^{n - l}_{\bar{w}_n}) = +$}\\
 			\chi^{n - l}_{\bar{w}_n}(\alpha, \beta)\to q^-_{\bar{w}_{n - l}},\text{ if $F(\chi^{n - l}_{\bar{w}_n}) = -$}	
 		\end{cases}
 	\end{align*}
 	for every $\chi \in \{\phi, \psi, \theta, \tau\}$. We show that these settings satisfy the ($\Leftarrow$)-part of the Lemma. 
 	
 	\textit{Case 1.1}. Let $\chi \in \{\phi, \psi, \theta, \tau\}$ and assume that $F(\chi^{n - l}_{\bar{w}_n}) = +$. Assume that the condition associated with $\chi^{n - l-1}_{\bar{w}_n}(\alpha, \beta)$ by the Lemma, holds in $\mathcal{M}'$. Then, by Induction Hypothesis, we must have $\mathcal{M}', \bar{w}_{n - l}\models_{\mathcal{L}} 	\chi^{n - l}_{\bar{w}_n}(\alpha, \beta)$. Now, let $v \in W^{un}_w$ be such that $\bar{w}_{n - l -1}\mathrel{\prec^{un}_w}v$ and $\mathcal{M}', v\models_{\mathcal{L}} 	q^+_{\bar{w}_{n - l}}$. By the definition of $[\mathcal{M}]$, it follows that $\bar{w}_{n - l}\mathrel{\prec^{un}_w}v$, but then we must also have $\mathcal{M}', v\models_{\mathcal{L}} 	\chi^{n - l}_{\bar{w}_n}(\alpha, \beta)$, by Lemma \ref{L:monotonicity}.
 	
 	\textit{Case 1.2}. Let $\chi \in \{\phi, \psi, \theta, \tau\}$ and assume that $F(\chi^{n - l}_{\bar{w}_n}) = -$. Assume that the condition associated with $\chi^{n - l-1}_{\bar{w}_n}(\alpha, \beta)$ by the Lemma, holds in $\mathcal{M}'$. Then, by Induction Hypothesis, we must have $\mathcal{M}', \bar{w}_{n - l}\not\models_{\mathcal{L}} 	\chi^{n - l}_{\bar{w}_n}(\alpha, \beta)$. Now, let $v \in W^{un}_w$ be such that $\bar{w}_{n - l -1}\mathrel{\prec^{un}_w}v$ and $\mathcal{M}', v\not\models_{\mathcal{L}} 	q^-_{\bar{w}_{n - l}}$. By the definition of $[\mathcal{M}]$, it follows that $v\mathrel{\prec^{un}_w}\bar{w}_{n - l}$, but then we must also have $\mathcal{M}', v\not\models_{\mathcal{L}} 	\chi^{n - l}_{\bar{w}_n}(\alpha, \beta)$, by Lemma \ref{L:monotonicity}.
 	
 	\textit{Case 2}. We have $w_{n - l -1}\mathrel{\succ}w_{n - l}$, thus also $\bar{w}_{n - l}\mathrel{\prec^{un}_w}\bar{w}_{n - l -1}$. We have then that $F(\phi^{n - l-1}_{\bar{w}_n}) = F(\psi^{n - l-1}_{\bar{w}_n}) = F(\theta^{n - l-1}_{\bar{w}_n}) = F(\tau^{n - l-1}_{\bar{w}_n}) := -$, and, as for the formulas, that
 	\begin{align*}
 		\chi^{n - l-1}_{\bar{w}_n}(\alpha, \beta) := \begin{cases}
 			q^+_{\bar{w}_{n - l}} \ll \chi^{n - l}_{\bar{w}_n}(\alpha, \beta),\text{ if $F(\chi^{n - l}_{\bar{w}_n}) = +$}\\
 			\chi^{n - l}_{\bar{w}_n}(\alpha, \beta)\ll q^-_{\bar{w}_{n - l}},\text{ if $F(\chi^{n - l}_{\bar{w}_n}) = -$}	
 		\end{cases}
 	\end{align*}
 	for every $\chi \in \{\phi, \psi, \theta, \tau\}$.  We show that these settings satisfy the ($\Leftarrow$)-part of the Lemma.
 	
 	\textit{Case 2.1}. Let $\chi \in \{\phi, \psi, \theta, \tau\}$ and assume that $F(\chi^{n - l}_{\bar{w}_n}) = +$. Assume that the condition associated with $\chi^{n - l-1}_{\bar{w}_n}(\alpha, \beta)$ by the Lemma, holds in $\mathcal{M}'$. Then, by Induction Hypothesis, we must have $\mathcal{M}', \bar{w}_{n - l}\models_{\mathcal{L}} 	\chi^{n - l}_{\bar{w}_n}(\alpha, \beta)$. Now, let $v \in W^{un}_w$ be such that $v\mathrel{\prec^{un}_w}\bar{w}_{n - l -1}$ and $\mathcal{M}', v\models_{\mathcal{L}} 	q^+_{\bar{w}_{n - l}}$. By the definition of $[\mathcal{M}]$, it follows that $\bar{w}_{n - l}\mathrel{\prec^{un}_w}v$, but then we must also have $\mathcal{M}', v\models_{\mathcal{L}} 	\chi^{n - l}_{\bar{w}_n}(\alpha, \beta)$, by Lemma \ref{L:monotonicity}.
 	
 	\textit{Case 2.2}. Let $\chi \in \{\phi, \psi, \theta, \tau\}$ and assume that $F(\chi^{n - l}_{\bar{w}_n}) = -$. Assume that the condition associated with $\chi^{n - l-1}_{\bar{w}_n}(\alpha, \beta)$ by the Lemma, holds in $\mathcal{M}'$. Then, by Induction Hypothesis, we must have $\mathcal{M}', \bar{w}_{n - l}\not\models_{\mathcal{L}} 	\chi^{n - l}_{\bar{w}_n}(\alpha, \beta)$. Now, let $v \in W^{un}_w$ be such that $v\mathrel{\prec^{un}_w}\bar{w}_{n - l -1}$ and $\mathcal{M}', v\not\models_{\mathcal{L}} 	q^-_{\bar{w}_{n - l}}$. By the definition of $[\mathcal{M}]$, it follows that $v\mathrel{\prec^{un}_w}\bar{w}_{n - l}$, but then we must also have $\mathcal{M}', v\not\models_{\mathcal{L}} 	\chi^{n - l}_{\bar{w}_n}(\alpha, \beta)$, by Lemma \ref{L:monotonicity}.
 \end{proof}

\begin{proposition}\label{L:saturation}
	Let $\mathcal{L}$ be an abstract bi-intuitionistic logic which extends $\mathsf{BIL}$, is
	preserved under bi-asimulations, $\star$-compact, and has TUP, and let $(\mathcal{M}, w) \in Pmod_\Theta$ be a bi-unravelled model. Then there exists a bi-unravelled model $(\mathcal{N}', w)\in Pmod_\Theta$ such that $\mathcal{M}\preccurlyeq_\mathcal{L} \mathcal{N}'$ and every $\mathcal{L}$-type of 
	$\mathcal{M}$ is realized in $\mathcal{N}'$.
\end{proposition}
\begin{proof}
	Assume the hypothesis of the Proposition. For the given $\mathcal{M}$, consider $[\mathcal{M}]$; we extend the signature of this model as
	follows:
	\begin{itemize}
		\item For every $\bar{v}_k \in W^{un}_w$, if $\Gamma, \Delta \subseteq
		L(\Theta_{\mathcal{M}})$ are such that
		$(\Gamma, \Delta)$ is a successor $\mathcal{L}$-type of
		$([\mathcal{M}],\bar{v}_k)$, then we add two fresh propositional letters $r^+_{\Gamma,\Delta,\bar{v}_k}$
		and $r^-_{\Gamma,\Delta,\bar{v}_k}$.
		
		\item For every $\bar{v}_k \in W^{un}_w$, if $\Gamma, \Delta \subseteq
		L(\Theta_{\mathcal{M}})$ are such that
		$(\Gamma, \Delta)$ is a predecessor $\mathcal{L}$-type of
		$([\mathcal{M}],\bar{v}_k)$, then we add two fresh propositional letters $s^+_{\Gamma,\Delta,\bar{v}_k}$
		and $s^-_{\Gamma,\Delta,\bar{v}_k}$.
	\end{itemize}
	We will call the resulting signature $\Theta'$. Consider, next, the following set of formulas:
	\begin{align*}
		\Upsilon = \{\phi^1_{\bar{v}_k}&(r^+_{\Gamma,\Delta,\bar{v}_k}, r^-_{\Gamma,\Delta,\bar{v}_k}),  \theta^1_{\bar{v}_k}(r^+_{\Gamma,\Delta,\bar{v}_k}, \gamma), \theta^1_{\bar{v}_k}(\delta, r^-_{\Gamma,\Delta,\bar{v}_k})\mid r^+_{\Gamma,\Delta,\bar{v}_k}, r^-_{\Gamma,\Delta,\bar{v}_k} \in \Theta',\,\gamma \in \Gamma,\,\delta \in \Delta\} \cup\\	
		&\cup\{\psi^1_{\bar{v}_k}(s^+_{\Gamma,\Delta,\bar{v}_k}, s^-_{\Gamma,\Delta,\bar{v}_k}),  \tau^1_{\bar{v}_k}(s^+_{\Gamma,\Delta,\bar{v}_k}, \gamma), \tau^1_{\bar{v}_k}(\delta, s^-_{\Gamma,\Delta,\bar{v}_k})\mid s^+_{\Gamma,\Delta,\bar{v}_k}, s^-_{\Gamma,\Delta,\bar{v}_k} \in \Theta',\,\gamma \in \Gamma,\,\delta \in \Delta\}
	\end{align*}
	We partition $\Upsilon$ into two subsets\footnote{Throughout this proof, we set that $F(\alpha):= F(\chi)$ if $\alpha$ is a substitution case of $\chi$.}, $\Upsilon^+ = \{\eta \in \Upsilon\mid F(\eta) = +\}$ and $\Upsilon^- = \{\eta \in \Upsilon\mid F(\eta) = -\}$, and consider the following $L(\Theta')$-theory $\Xi = (\Xi^+, \Xi^-) = (Th^+_\mathcal{L}([\mathcal{M}], w) \cup \Upsilon^+, Th^-_\mathcal{L}([\mathcal{M}], w) \cup \Upsilon^-)$. 
	
	We show that $\Xi$ is finitely $\mathcal{L}$-satisfiable. Indeed, given an arbitrary $(S_0, T_0) \Subset \Xi$
	we know, wlog, that, for some $n < \omega$ we have:
	$$
	(S_0, T_0)  = (S_1, T_1)  \cup \ldots \cup (S_n, T_n),
	$$
	where for every $1 \leq i \leq n$ one of the following two cases
	holds:
	
	\emph{Case 1}. For some  $\bar{v}_k \in W^{un}_w$, there exist  $\Gamma, \Delta \subseteq
	L(\Theta_{\mathcal{M}})$ such that
	$(\Gamma, \Delta)$ is a successor $\mathcal{L}$-type of
	$([\mathcal{M}],\bar{v}_k)$ and $\Gamma'
	\Subset\Gamma$, $\Delta' \Subset \Delta$ such that we have: $(S_i, T_i) = (S'_i\cup S''_i, T'_i\cup T''_i)$, where $(S''_i, T''_i) \Subset (Th_\mathcal{L}([\mathcal{M}], w)$ and: 
	\begin{align*}
	S'_i\cup T'_i  &\subseteq \{\phi^1_{\bar{v}_k}(r^+_{\Gamma,\Delta,\bar{v}_k}, r^-_{\Gamma,\Delta,\bar{v}_k}),  \theta^1_{\bar{v}_k}(r^+_{\Gamma,\Delta,\bar{v}_k}, \gamma), \theta^1_{\bar{v}_k}(\delta, r^-_{\Gamma,\Delta,\bar{v}_k})\mid \gamma \in \Gamma', \delta \in \Delta'\}
	\end{align*}
	By Definition \ref{D:types}, $(\Gamma', \Delta')$ is
	$\mathcal{L}$-satisfied at $([\mathcal{M}], \bar{u}_m)$
	for some $\bar{u}_m \in W^{un}_w$ such that $\bar{v}_k\mathrel{\prec^{un}_w}\bar{u}_m$. But then it follows from Lemma \ref{L:representation-m}, that $(S_i, T_i)$ must be $\mathcal{L}$-satisfied at $w$ in the expansion $\mathcal{M}'$ of
	$[\mathcal{M}]$ in which $r^+_{\Gamma,\Delta,\bar{v}_k}$ and  $r^-_{\Gamma,\Delta,\bar{v}_k}$ are identified with $q^+_{\bar{u}_m}$, $q^-_{\bar{u}_m}$, respectively.

	\emph{Case 2}. For some  $\bar{v}_k \in W^{un}_w$, there exist  $\Gamma, \Delta \subseteq
L(\Theta_{\mathcal{M}})$ such that
$(\Gamma, \Delta)$ is a predecessor $\mathcal{L}$-type of
$([\mathcal{M}],\bar{v}_k)$ and $\Gamma'
\Subset\Gamma$, $\Delta' \Subset \Delta$ such that we have: $(S_i, T_i) = (S'_i\cup S''_i, T'_i\cup T''_i)$, where $(S''_i, T''_i) \Subset (Th_\mathcal{L}([\mathcal{M}], w)$ and: 
\begin{align*}
	S'_i\cup T'_i  &\subseteq \{\psi^1_{\bar{v}_k}(s^+_{\Gamma,\Delta,\bar{v}_k}, s^-_{\Gamma,\Delta,\bar{v}_k}),  \tau^1_{\bar{v}_k}(s^+_{\Gamma,\Delta,\bar{v}_k}, \gamma), \tau^1_{\bar{v}_k}(\delta, s^-_{\Gamma,\Delta,\bar{v}_k})\mid \gamma \in \Gamma', \delta \in \Delta'\}
\end{align*}
By Definition \ref{D:types}, $(\Gamma', \Delta')$ is
$\mathcal{L}$-satisfied at $([\mathcal{M}], \bar{u}_m)$
for some $\bar{u}_m \in W^{un}_w$ such that $\bar{u}_m\mathrel{\prec^{un}_w}\bar{v}_k$. But then it follows from Lemma \ref{L:representation-m}, that $(S_i, T_i)$ must be $\mathcal{L}$-satisfied at $w$ in the expansion $\mathcal{M}'$ of
$[\mathcal{M}]$ in which $s^+_{\Gamma,\Delta,\bar{v}_k}$ and  $s^-_{\Gamma,\Delta,\bar{v}_k}$ are identified with $q^+_{\bar{u}_m}$, $q^-_{\bar{u}_m}$, respectively.	 
 
	Note, moreover, that for all $1 \leq i < j \leq n$, the set of
	propositional letters that needs to be added to $[\mathcal{M}]$ in
	order to get $T_i$ satisfied at $w$ is disjoint from the set of
	propositional letters to be added to this same model in order to get $T_j$
	satisfied at $w$. Therefore, we can take the union of the
	expansions required by $T_1, \ldots, T_n$ and get (by Expansion property) an expansion
	$\mathcal{M}'$ of $[\mathcal{M}]$ such that $(S_0,
	T_0)$ is satisfied at $(\mathcal{M}', w)$. Since $(S_0,
	T_0) \Subset \Xi$ was chosen arbitrarily, this means, by the
	$\star$-compactness of $\mathcal{L}$, that
	$\Xi$ itself is
	$\mathcal{L}$-satisfiable.

	Let $(\mathcal{M}_1, w_1)$ be a pointed $\Theta'$-model
	$\mathcal{L}$-satisfying
	$\Xi$. We know, by Lemma
	\ref{L:unravelling-asim} and the preservation of $\mathcal{L}$ under bi-asimulations, that $((\mathcal{M}_1)^{un}_{w_1}, w_1)$ also
	$\mathcal{L}$-satisfies $\Xi$. We now set $\mathcal{M}': = (\mathcal{M}_1)^{un}_{w_1}\upharpoonright\Theta_{\mathcal{M}}$. It follows, by Expansion Property, that $\mathcal{M}'$ $\mathcal{L}$-satisfies $Th_\mathcal{L}([\mathcal{M}], w)$.
	Therefore, by Lemma \ref{L:lemma1}, there must be an
	$\mathcal{L}$-elementary embedding $f$ of $[\mathcal{M}]$ into
	$\mathcal{M}'$, and for this elementary embedding we will have
	with $f(w) = w_1$. Note, moreover, that, by Lemma \ref{L:embedding}.5, we have then $f([\mathcal{M}]) \preccurlyeq_\mathcal{L} \mathcal{M}'$ where $f([\mathcal{M}])$ is an $f$-isomorphic copy of $[\mathcal{M}]$. Therefore, by Lemma  \ref{L:embedding}.6, there must exist a $\Theta_{\mathcal{M}}$-model $\mathcal{N}$ and a function $f' \supseteq f$, such that $[\mathcal{M}] \preccurlyeq_\mathcal{L} \mathcal{N}$, and $f': \mathcal{N} \cong \mathcal{M}'$. We now prove the following:
	
	\emph{Claim}. $\mathcal{N}$ realizes every $\mathcal{L}$-type
	of $[\mathcal{M}]$.
	
	To prove this Claim, we again have to consider the two cases
	outlined above.
	
	\emph{Case 1}. For some  $\bar{u}_r \in W^{un}_w$, there exist  $\Gamma, \Delta \subseteq
	L(\Theta_{\mathcal{M}})$ such that
	$(\Gamma, \Delta)$ is a successor $\mathcal{L}$-type of
	$([\mathcal{M}],\bar{u}_r)$. Note that for every formula $\chi$ in the set:
	$$
	\{\phi^1_{\bar{u}_r}(r^+_{\Gamma,\Delta,\bar{u}_r}, r^-_{\Gamma,\Delta,\bar{u}_r}),  \theta^1_{\bar{u}_r}(r^+_{\Gamma,\Delta,\bar{u}_r}, \gamma), \theta^1_{\bar{u}_r}(\delta, r^-_{\Gamma,\Delta,\bar{u}_r})\mid \gamma \in \Gamma,\,\delta \in \Delta\}
	$$
	we have that $(\mathcal{M}_1)^{un}_{w_1}, w_1 \models_\mathcal{L} \chi$ iff $F(\chi) = +$. Therefore, since we also have $w_1 = f(w) = f(u_1)$ (by $\bar{u}_r \in W^{un}_w$ and the definition of bi-unravelling), it follows from Lemma \ref{L:representation-emb} that for some $v \in (W_1)^{un}_{w_1}$ we have both $f(\bar{u}_r)\mathrel{(\prec_1)^{un}_{w_1}}v$ and $(\mathcal{M}_1)^{un}_{w_1}, v\models_\mathcal{L}(\{r^+_{\Gamma,\Delta,\bar{u}_r}\},\{r^-_{\Gamma,\Delta,\bar{u}_r}\})$. On the other hand, it follows from the same  Lemma \ref{L:representation-emb} that for all $\gamma \in \Gamma$ and $\delta \in \Delta$ we have:
	$$
	(\mathcal{M}_1)^{un}_{w_1}, f(\bar{u}_r) \models_\mathcal{L} (r^+_{\Gamma,\Delta,\bar{u}_r}\to \gamma) \wedge (\delta\to r^-_{\Gamma,\Delta,\bar{u}_r}). 
	$$ 
	But then, it follows, by the choice of $v$, that we have:
	$$
	(\mathcal{M}_1)^{un}_{w_1}, v\models_\mathcal{L}(\Gamma,\Delta).
	$$
	Since we have $\mathcal{M}'= (\mathcal{M}_1)^{un}_{w_1}\upharpoonright\Theta_{\mathcal{M}}$, it follows that also $f(\bar{u}_r)\mathrel{\prec'}v$, and, by Expansion Property, that:
	$$
	\mathcal{M}', v\models_\mathcal{L}(\Gamma,\Delta).
	$$
	But then consider $u \in U$ such that $f'(u) = v$. We have then $\bar{u}_r\mathrel{\lhd}u$, and, by Lemma \ref{L:embedding}.4, that 
	$$
	\mathcal{N}, u\models_\mathcal{L}(\Gamma,\Delta),
	$$
	so that $\mathcal{N}$ realizes $(\Gamma,\Delta)$.
	
	\emph{Case 2}. For some  $\bar{u}_r \in W^{un}_w$, there exist  $\Gamma, \Delta \subseteq
	L(\Theta_{\mathcal{M}})$ such that
	$(\Gamma, \Delta)$ is a predecessor $\mathcal{L}$-type of
	$([\mathcal{M}],\bar{u}_r)$. Note that for every formula $\chi$ in the set:
	$$
	\{\psi^1_{\bar{u}_r}(s^+_{\Gamma,\Delta,\bar{u}_r}, s^-_{\Gamma,\Delta,\bar{u}_r}),  \tau^1_{\bar{u}_r}(s^+_{\Gamma,\Delta,\bar{u}_r}, \gamma), \tau^1_{\bar{u}_r}(\delta, s^-_{\Gamma,\Delta,\bar{u}_r})\mid \gamma \in \Gamma,\,\delta \in \Delta\}
	$$
	we have that $(\mathcal{M}_1)^{un}_{w_1}, w_1 \models_\mathcal{L} \chi$ iff $F(\chi) = +$. Therefore, since we also have $w_1 = f(w) = f(u_1)$ (by $\bar{u}_r \in W^{un}_w$ and the definition of bi-unravelling), it follows from Lemma \ref{L:representation-emb} that for some $v \in (W_1)^{un}_{w_1}$ we have both $v\mathrel{(\prec_1)^{un}_{w_1}}f(\bar{u}_r)$ and $(\mathcal{M}_1)^{un}_{w_1}, v\models_\mathcal{L}(\{s^+_{\Gamma,\Delta,\bar{u}_r}\},\{s^-_{\Gamma,\Delta,\bar{u}_r}\})$. On the other hand, it follows from the same  Lemma \ref{L:representation-emb} that for all $\gamma \in \Gamma$ and $\delta \in \Delta$ we have:
	$$
	(\mathcal{M}_1)^{un}_{w_1}, f(\bar{u}_r) \not\models_\mathcal{L} (s^+_{\Gamma,\Delta,\bar{u}_r}\ll \gamma) \vee (\delta\ll s^-_{\Gamma,\Delta,\bar{u}_r}). 
	$$ 
	But then, it follows, by the choice of $v$, that we have:
	$$
	(\mathcal{M}_1)^{un}_{w_1}, v\models_\mathcal{L}(\Gamma,\Delta).
	$$
	Since we have $\mathcal{M}'= (\mathcal{M}_1)^{un}_{w_1}\upharpoonright\Theta_{\mathcal{M}}$, it follows that also $v\mathrel{\prec'}f(\bar{u}_r)$, and, by Expansion Property, that:
	$$
	\mathcal{M}', v\models_\mathcal{L}(\Gamma,\Delta).
	$$
	But then consider $u \in U$ such that $f'(u) = v$. We have then $u\mathrel{\lhd}\bar{u}_r$, and, by Lemma \ref{L:embedding}.4, that 
	$$
	\mathcal{N}, u\models_\mathcal{L}(\Gamma,\Delta),
	$$
	so that $\mathcal{N}$ realizes $(\Gamma,\Delta)$.
	
	Our Claim is thus proven.
	
	It remains to notice that, by Lemma \ref{L:type-realization}.2, $\mathcal{N}':= \mathcal{N}\upharpoonright\Theta$ must realize every $\mathcal{L}$-type of $\mathcal{M} = [\mathcal{M}]\upharpoonright\Theta$, and that, by $f':\mathcal{N} \cong \mathcal{M}'$, $\mathcal{N}'$ is a $\Theta$-reduct of a bi-unravelled model and hence a bi-unravelled $\Theta$-model itself. 
\end{proof}

\begin{corollary}\label{C:saturation}
	Let $\mathcal{L}$ be an abstract bi-intuitionistic logic which extends $\mathsf{BIL}$, is
	preserved under bi-asimulations, $\star$-compact, and has TUP, and let $(\mathcal{M}, w)\in Pmod_\Theta$. Then there exists a $(\mathcal{N}, w) \in Pmod_\Theta$ such that $Th_\mathcal{L}(\mathcal{M}, w) = Th_\mathcal{L}(\mathcal{N}, w)$ and $\mathcal{N}$ is $\mathcal{L}$-saturated.	
\end{corollary}
\begin{proof}
	It follows from invariance of $\mathcal{L}$ under bi-asimulations and Lemma \ref{L:unravelling-asim} that $Th_\mathcal{L}(\mathcal{M}, w) = Th_\mathcal{L}(\mathcal{M}^{un}_w, w)$. Next, applying Proposition \ref{L:saturation} $\omega$ times, we form an $\mathcal{L}$-elementary chain of submodels
	$$
	\mathcal{M}^{un}_w = \mathcal{N}_1 \preccurlyeq_\mathcal{L}\ldots \preccurlyeq_\mathcal{L} \mathcal{N}_n \preccurlyeq_\mathcal{L}\ldots
	$$
	such that for every $i > 0$, we have $(\mathcal{N}_i, w) \in Str_\mathcal{L}(\Theta)$ and $\mathcal{N}_{i + 1}$ realizes every $\mathcal{L}$-type of $\mathcal{N}_{i}$. We then set $\mathcal{N} := \bigcup_{i > 0}\mathcal{N}_{i}$ so that $(\mathcal{N}, w) \in Str_\mathcal{L}(\Theta)$. Since $\mathcal{L}$ has TUP, we have then $\mathcal{N}_{i} \preccurlyeq_\mathcal{L} \mathcal{N}$ for every $i > 0$. In particular, we have $\mathcal{M}^{un}_w = \mathcal{N}_1 \preccurlyeq_\mathcal{L} \mathcal{N}$, so that  $Th_\mathcal{L}(\mathcal{M}, w) = Th_\mathcal{L}(\mathcal{N}, w)$ follows. It remains to show the $\mathcal{L}$-saturation of $\mathcal{N}$, that is to say, that $\mathcal{N}$ realizes all of its $\mathcal{L}$-types. Since we have two sorts of types in $\mathcal{L}$, we have to consider two possible cases.
	
	\textit{Case 1}. For some $v \in U$ the sets $\Gamma, \Delta \subseteq
	L(\Theta)$ are such that
	$(\Gamma, \Delta)$ is a successor $\mathcal{L}$-type of
	$(\mathcal{N}, v)$. Then, by Corollary \ref{C:types-formulas}.1, for all $\Gamma'\Subset \Gamma$ and $\Delta' \Subset \Delta$, we have $\bigwedge\Gamma'\to \bigvee\Delta' \in Th_{\mathcal{L}}(\mathcal{N}, v)$. But then, we can choose a $j > 0$ such that $\mathcal{N}_{j}$ is the first model in the chain for which we have $v \in U_j$. Since we also have $\mathcal{N}_{j} \preccurlyeq_\mathcal{L} \mathcal{N}$, it follows that $Th_{\mathcal{L}}(\mathcal{N}_j, v) = Th_{\mathcal{L}}(\mathcal{N}, v)$ and thus, for all $\Gamma'\Subset \Gamma$ and $\Delta' \Subset \Delta$, we have $\bigwedge\Gamma'\to \bigvee\Delta' \in Th_{\mathcal{L}}(\mathcal{N}_j, v)$. But then, by Corollary \ref{C:types-formulas}.1, 	$(\Gamma, \Delta)$ must be a successor $\mathcal{L}$-type of
	$(\mathcal{N}_j, v)$, and, therefore, $(\Gamma, \Delta)$ must be realized in $\mathcal{N}_{j + 1}$. But, since we also have $\mathcal{N}_{j + 1}\preccurlyeq_\mathcal{L} \mathcal{N}$, it follows from Lemma \ref{L:type-realization}.1, that $(\Gamma, \Delta)$ also must be realized in $\mathcal{N}$ itself.
	
	\textit{Case 2}, where we assume that our $\mathcal{L}$-type is a predecessor $\mathcal{L}$-type, is solved in the same manner. The only difference from Case 1 is that we need to apply Corollary \ref{C:types-formulas}.2 in place of Corollary \ref{C:types-formulas}.1. 	 
\end{proof}

We are now in a position to prove Theorem \ref{L:main}. Indeed,
assume the hypothesis of the theorem, and assume, for
contradiction, that $\mathcal{L}\not\equiv\mathcal{L}'$. By
Proposition \ref{L:proposition1}, there must be a $\phi \in
L(\Theta_\phi)$ and
$(\mathcal{M}_1, w_1)$, $(\mathcal{M}_2, w_2) \in Pmod_{\Theta_\phi}$ such that
$Th^+_{BIL}(\mathcal{M}_1, w_1) \subseteq Th^+_{BIL}(\mathcal{M}_2,
w_2)$ while $\mathcal{M}_1, w_1 \models_\mathcal{L} \phi$ and
$\mathcal{M}_2, w_2 \not\models_\mathcal{L} \phi$. By Corollary
\ref{C:saturation}, take $\mathcal{L}$-saturated
$\Theta_\phi$-models $\mathcal{N}_1$ and $\mathcal{N}_2$ such that $Th_\mathcal{L}(\mathcal{M}_i, w_i) =
Th_\mathcal{L}(\mathcal{N}_i, w_i)$ for $i \in \{ 1,2 \}$. We
will have then, of course, that $\mathcal{N}_1, w_1
\models_\mathcal{L} \phi$, but $\mathcal{N}_2, w_2
\not\models_\mathcal{L} \phi$. On the other hand, we will still
have
\[
Th^+_{\mathcal{L}'}(\mathcal{N}_1, w_1) \subseteq
Th^+_{\mathcal{L}'}(\mathcal{N}_2, w_2),
\]
whence, by Corollary
\ref{L:asimulationscorollary}, there must be a bi-asimulation $A$
from $(\mathcal{N}_1, w_1)$ to $(\mathcal{N}_2, w_2)$, but then,
since $\mathcal{L}$ is preserved under bi-asimulations, we must also
have \[
Th^+_\mathcal{L}(\mathcal{N}_1, w_1) \subseteq
Th^+_\mathcal{L}(\mathcal{N}_2, w_2).
\]
Now, since $\phi \in
Th^+_\mathcal{L}(\mathcal{N}_1, w_1)$, we can see that we must also have $\phi \in
Th^+_\mathcal{L}(\mathcal{N}_2, w_2)$, so that $\mathcal{N}_2, w_2
\models_\mathcal{L} \phi$, which is a contradiction.

\section{Conclusion}
In this paper  we have extended the methods of \cite{baok} to the case of bi-intuitionistic logic, showing bi-intuitionistic propostional logic to be the strongest abstract logic satisfying the Tarski Union Property, a variation of compactness and preservation under bi-asimulations. This extension is rather non-trivial given the complexity of the unravelling construction in the bi-intuitionistic context and the added complexity in building the desired saturated structures required for the proof to go through. It would be interesting to find other Lindstr\"om characterizations of bi-intuitionistic propositional logic (perhaps finding some natural property that can replace the Tarski Union Property) but we will leave this task for another occasion, as we do not know how to proceed.

%

}
\end{document}